\UseRawInputEncoding

 \documentclass[a4paper]{amsart}

\usepackage{amsmath}
\usepackage{amssymb}
\usepackage{mathrsfs}
\usepackage[all]{xy}
\usepackage{mathtools}
\usepackage{color}
\usepackage{graphicx}
\usepackage{float}
\usepackage{array,arydshln}
\usepackage{amsthm}
\usepackage{comment}

\newcommand{\Address}{{
Tokio Sasaki, \textsc{Department of Mathematics, University of Miami, 1365 Memorial Drive, Coral Gables, FL 33146}\par\nopagebreak
  \textit{E-mail address}: \texttt{sasakitokio@math.miami.edu}
  }}
\newtheorem{thm}{Theorem}
\newtheorem{lem}{Lemma}
\newtheorem{prop}{Proposition}
\newtheorem{cor}{Corollary}
\newtheorem*{conj}{Conjecture}

\theoremstyle{definition}

\theoremstyle{remark}
\newtheorem*{remark}{Remark}
\newtheorem*{notation}{Notation}

\numberwithin{thm}{section}
\numberwithin{lem}{section}
\numberwithin{prop}{section}
\numberwithin{cor}{section}
\numberwithin{definition}{section}

\newcommand{\twist}{2 \pi i}

\newcommand{\Li}{{\rm Li}}
\newcommand{\Hdg}{{\rm Hdg}}
\newcommand{\Hom}{{\rm Hom}}
\newcommand{\MHS}{{\rm MHS}}
\newcommand{\VMHS}{{\rm VMHS}}
\newcommand{\MHM}{{\rm MHM}}
\newcommand{\Ker}{{\rm Ker}}
\newcommand{\Coker}{{\rm Coker}}
\newcommand{\Ext}{{\rm Ext}}
\newcommand{\Gr}{{\rm Gr}}
\newcommand{\Pic}{{\rm Pic}}
\newcommand{\dec}{{\rm dec}}
\newcommand{\ind}{{\rm ind}}
\newcommand{\Griff}{{\rm Griff}}

\begin{document}

\begin{abstract}
	We find a new method to detect the linearly independence of $\mathbb{R}$-regulator indecomposable $K_1$-cycles which is based on the singularities and limits of admissible normal functions. We also construct a collection of higher Chow cycles on certain surfaces in $\mathbb{P}^3$ of degree $d \ge 4$ which degenerate to an arrangement of $d$ planes in general position. By applying our method, we show that these higher Chow cycles are enough to show the surjectivity of the real regulator map when $d = 4$. Hence our construction gives a new explicit proof of the Hodge-$\mathcal{D}$-Conjecture for a certain type of $K3$ surfaces. As an application, we also construct a general semistable degeneration family of degree $d + 1$ threefolds in $\mathbb{P}^3 \times \mathbb{P}^1 \times \mathbb{P}^1$ such that a codimension 1 stratum is a surface of the above type. The real regulator indecomposability of our higher Chow cycles implies that the Griffiths group of  the general fiber of these threefolds is non trivial.
\end{abstract}

\title{Limits and Singularities of Normal Functions}

\author{Tokio Sasaki}

\maketitle

\section{Introduction}
On a smooth projective variety $X$ over $\mathbb{C}$, finding interesting cycles is one of the central themes of Algebraic Geometry.
Especially, the celebrated Hodge Conjecture states the surjectivity of the cycle class map from the Chow group $CH^p(X)$ with the rational coefficients to the Hodge cycle class $\Hdg(X)$.
By changing the Chow group to the higher Chow groups $CH^p(X,n)$ with real coefficients and the cycle map to the real regulator map $r_{\mathcal{D}, \mathbb{R}}^{p,n} \colon CH^p(X,n)\otimes \mathbb{R} \to H^{2p-n}_{\mathcal{D}}(X, \mathbb{R}(p))$, 
we can generalize the Hodge Conjecture to the Hodge-$\mathcal{D}$-Conjecture, which states the surjectivity of $r_{\mathcal{D}, \mathbb{R}}^{p,n}$.
Unfortunately this conjecture is false for general projective varieties, but it is still open (and expected to hold) for $X$ defined over $\overline{\mathbb{Q}}$ (\cite{MS97}).
The most significant result is due to X. Chen and J. Lewis. They proved that the Hodge-$\mathcal{D}$-Conjecture holds for (analytically) general polarized $K3$ surfaces in the moduli space by observing the deformation of the higher cycles along the degeneration of the general $K3$ to a special one with Picard number 20 (which is called Bryan-Leung $K3$ surface). See the introduction of \cite{CL05}  for more historical detail and results.

The aim of this paper is to introduce a a new method to detect the linear independence of indecomposable $K_1$-cycles on the nearby fiber by using the asymptotic behaviors of real regulators. We also find a systematic construction of a collection of higher Chow cycles on a certain type of families $\{ X_t \}_{t \in \mathbb{P}^1}$ of degree $d$ surfaces in $\mathbb{P}^3$ such that our method can show that these higher cycles are sufficient to span the Deligne cohomology under $r_{\mathcal{D}, \mathbb{R}}^{2,1}$ when $d = 4$.
While the existence of such cycles is abstractly contained in Chen and Lewis's work, our construction is completely explicit and concrete.
The precise construction is given in Section \ref{Construction of Families of Higher Cycles}, but roughly the type of surfaces we consider has the form
\[X_t \colon L_1L_2 \cdots L_d + tM_1M_2 \cdots M_d =0  \subset \mathbb{P}^3\]
with general $t \in \mathbb{P}^1$ and linear forms $L_i, M_l$ in general position. Then each intersection $L_i \cap M_l$ defines a line on $X_t$, which is constant even when we move $t$.
By choosing three lines of this type with the boundaries at $L_i \cap L_j \cap M_l$, we can construct a higher Chow cycle $\gamma_{ijk,l} \in CH^2(X_t, 1)$ so that its support is just a union of three lines.
By changing the roles of the linear forms $L$ and $M$, we also can construct another type of higher Chow cycle $\delta_{i, lmn}$, and moreover each line $L_i \cap M_l$ as an algebraic cycle also defines an element $\lambda_{il}$ of $CH^2(X_t, 1)$ in the naive way.
Theorem \ref{Hodge-D-Conjecture} states that these higher cycles $\{ \gamma_{ijk,l} \}, \{\delta_{i, lmn}\}, \{\lambda_{il}\}$ are enough to prove the Hodge-$\mathcal{D}$-Conjecture for $d = 4$ and general choices of $t, L_i, M_s$.

Our method is based on the theory of limits and singularities of admissible normal functions.
After a resolution of singularities and change of the coordinates, we may consider $\{ X_t \}_{t \in \Delta^*}$ as a semistable degeneration to the simple normal crossing divisor $X_0$ with smooth fibers over the punctual unit disc $\Delta^* = \Delta \setminus \{0\}$.
The Abel-Jacobi values of $\gamma_{ijk,l}$ and $\delta_{i, lmn}$ as families of higher Chow cycles define holomorphic sections of the intermediate Jacobian bundle over $\Delta^*$, which are examples of admissible normal functions.
Roughly speaking, the \textit{limit} of the admissible normal function associated to a family of higher Chow cycles describes the limiting behavior of the Abel-Jacobi value as $t$ approaches to $0$.
However, generally the degeneration of the family of higher Chow cycles may not be a higher Chow cycle, since it may have some obstructions coming from the singularities.
Such an obstruction can be described as another invariant, which is called the \textit{singularity} of the admissible normal function.

Recently, the limiting behaviors of complex valued admissible (or usual) normal functions has been studied ( \cite{GGK10} \cite{DK11} \cite{dDI17}) and it is not difficult to show that the singularity invariants factor through the projection to the real regulator.
In fact, for general $d$ we show that each $\gamma_{ijk,l}$ has non-trivial singularities and moreover $\{ \gamma_{ijk,l} \} \cup \{\lambda_{il}\}$ span the codomain $\Hdg(\Coker N)$ of this invariant, where $N$ denotes the log monodromy action around $t = 0$ (Theorem \ref{main thm}) and this implies that $\{\gamma_{ijk,l}\}$ span a 19 dimensional subspace of $H^{1,1}_{\mathbb{R}}$.

On the other hand, the limit invariants are typically killed by the projection to the real regulator. As our new method, in Section \ref{Higher Chow Cycles with Non-trivial Limits} and \ref{Application: Hodge-D-Conjecture for a certain type of K3 surfaces} we dig further into the asymptotic behavior of the real regulator to recover these limits.
We show not only the non-triviality of the limit of $\delta_{i, lmn}$ when $d=4$, but also the limit of its real regulator value is linearly independent from that of $\{ \frac{1}{\log (t)} \gamma_{ijk,l} \}$.
These results give explicit proof of the Hodge-$\mathcal{D}$-Conjecture for this type of $K3$ surface.
We also remark that our construction itself yields a collection of higher cycles on the surface $X_t$ of general degree $d \ge 4$. While the Hodge-$\mathcal{D}$-Conjecture is known to be false for very general surfaces in $\mathbb{P}^3$ of degree $ \ge 5$, it is still an interesting problem to determine subfamilies on which the conjecture holds.
Though we are not able to prove Hodge-$\mathcal{D}$-Conjecture yet for $X_t$ of general degree,  at lease each of $\gamma_{ijk,l}$ is still $\mathbb{R}$-regulator indecomposable (cf. Section \ref{Higher Chow Groups and Hodge-D-Conjecture} for the definition). It suggests that the higher cycles $\{ \gamma_{ijk,l} \}, \{\delta_{i, lmn}\}, \{\lambda_{il}\}$ may indicate an explicit proof.

As an application, the higher cycle $\gamma_{ijk,l}$ also gives a new construction of threefolds with non-trivial Griffiths groups.
The Griffiths group $\Griff^p(X)$ of a projective variety $X$ is defined by the quotient $CH^p_{hom}(X) / CH^p_{alg}(X)$ by the subgroup $CH^p_{alg}(X)$ of cycles which are algebraically equivalent to zero.
Cycles in $\Griff^2(X)$ and their normal functions provide the B-model for Morrison and Walcher's work on the open mirror symmetry (\cite{MW09}).
Meanwhile, C. Doran, A. Harder and A. Thompson introduced a non-toric mirror scenario involving Tyurin degenerations, in which the Calabi-Yau threefolds degenerate to a union of quasi-Fano threefolds intersecting along a $K3$ surface (\cite{DHT17}).
Key to studying open mirror symmetry in the latter setting would be to construct $K_1$-cycles on the $K3$ surface which are limits of $K_0$-cycles on the nearby Calabi-Yau threefold (this is called ``going-up'' in the theory of the $K$-theory elevator which is introduced in \cite{dDI17}).
For this construction, one will need totally concrete $K_1$-cycles on the $K3$ surface, and this point is an advantage of our explicit proof of the existence of $\mathbb{R}$-regulator indecomposable cycles.
In Section \ref{Application: Threefold with Non-Trivial Griffiths Groups}, starting from a general degree $d$ surface $X_{t_0}$ defined as above, we construct a semistable degeneration family $\mathcal{Y}$ of threefolds, which is an example of Tyurin degeneration when $d = 4$.
Its singular fiber $Y_0$ consists of the union of the product $X_{t_0} \times \mathbb{P}^1$ and two blown up copies of $\mathbb{P}^3$, meeting along two copies of $X_{t_0}$ (The picture before taking the blow up is drawn in Figure 2).

Applying the theory of the $K$-theory elevator, we can shift the higher Chow cycle $ \gamma_{ijk,l} $  in the intersection $X_{t_0} \times \mathbb{P}^1$  of $Y_0$ to an algebraic cycle in one of the blown up $\mathbb{P}^3$, which is a fiber of a family of algebraic cycles $\mathcal{C}_ {ijk,l}$ on $\mathcal{Y}$.
$\mathbb{R}$-regulator indecomposability of $\gamma_{ijk,l}$ implies the non-triviality of the general fiber of $\mathcal{C}_ {ijk,l}$  in the Griffiths groups.
Therefore this yields a new example exhibiting the connection between the algebraically non-trivial cycles and $\mathbb{R}$-regulator indecomposable cycles.

This paper is organized as follows.
In Section \ref{Higher Chow Groups and Hodge-D-Conjecture}, we briefly recall the definitions of the higher Chow cycles, indecomposable cycles, real regulator map, and the statement of Hodge-$\mathcal{D}$-conjecture. We also introduce the KLM formula, which is an essential tool in computing the Abel-Jacobi maps.
In Section \ref{Construction of Families of Higher Cycles}, we define the family of surfaces $\mathcal{X} = \{ X_t \}$ and construct the specific higher Chow cycles $\gamma_{ijk,l}$, $\delta_{i, lmn}$, and $\lambda_{il}$ on this family.
Before entering the proof of the main theorem, we introduce the definition and the general discussion about the limits and singularities of higher normal functions in Section \ref{Singularities and limits of Normal Functions}.
Then, in Section \ref{Higher Chow Cycles with Non-trivial Singularities} and Section \ref{Higher Chow Cycles with Non-trivial Limits} we show the non-triviality of these invariants for the above higher Chow cycles respectively, and prove the Hodge-$\mathcal{D}$-Conjecture for our case in Section \ref{Application: Hodge-D-Conjecture for a certain type of K3 surfaces}.
Finally, we construct a threefold with non-trivial Griffiths groups starting from $X_{t_0}$ as another application in Section \ref{Application: Threefold with Non-Trivial Griffiths Groups}.

\subsection*{Acknowledgements}
The author acknowledge support under NSF FRG grant [DMS-1361147; PI: Matt Kerr], and wish to thank his advisor Matt Kerr for helpful discussions and encouragement. He also would like to thank Tomohide Terasoma for firstly suggesting him to study the higher cycles on this type of $K3$ surface.

\section{Higher Chow Groups and Hodge-$\mathcal{D}$-Conjecture}
\label{Higher Chow Groups and Hodge-D-Conjecture}
Throughout this paper, we fix the base field to be $\mathbb{C}$ and an algebraic variety means an integral separated scheme of finite type over $\mathbb{C}$.
Firstly, we recall the definition of the higher Chow groups.
See \cite{B86} for the original construction with algebraic simplexes and \cite{Le94} for the cubical version, which we use here.
The algebraic $n$-cube is defined by
\[
\Box^n \coloneqq (\mathbb{P}^1 \setminus \{ 1 \})^n.
\]
For each $i$ $(0 \leq i \leq n)$, there is the $i$-th-face map $\rho_i^{\epsilon} \colon \Box^{n-1} \hookrightarrow \Box^{n}$ with $\epsilon = 0, \infty$ defined by the embedding $(z_1, z_2, \ldots, z_{n-1}) \mapsto (z_1, z_2, \ldots, z_{i-1}, \epsilon, z_i, \ldots, z_n)$. The facet $\partial_i^{\epsilon} \Box^n$ is defined by the image of $\rho^{\epsilon}_i $ and more generally the face $\partial^{\underline{\epsilon}}_I \Box^n$ for each $I \subset \left\{ 0, \ldots, n \right\}$ and $\underline{\epsilon} = \{\epsilon(i)\}_{i \in I}$ is defined by $\bigcap_{i \in I}\partial_i^{\epsilon(i)}\Box^n$.
We also denote $\partial \Box^n \coloneqq \bigcap_{i \in I } \bigcup_{\epsilon = 0, \infty} \partial^{\epsilon}_i \Box^n$.

Let $U$ be a quasi-projective variety.
For $p, n \in \mathbb{Z}_{\geq 0}$, $\mathcal{C}^p(U, n)$ is defined as the free abelian group generated by subvarieties of $U \times \Box^n$ of codimension $p$ which intersects each $U  \times \partial^{\underline{\epsilon}}_I \Box^n$ properly.
It contains the subgroup $\mathcal{D}^p(U , n)$ which is generated by the pullbacks of cycles via face projections $U \times \Box^n \twoheadrightarrow U \times \Box^{n- |I|}$, and we denote the quotient $\mathcal{C}^p(U, n) / \mathcal{D}^p(U, n)$ by $\mathcal{Z}^p(U, n)$.
 Then $\mathcal{Z}^p(U, \bullet)$ becomes a chain complex with the well-defined boundary map
\[
\partial \coloneqq \sum_{i} (-1)^i ((\rho^0_i)^* - (\rho^{\infty}_i)^*) \colon \mathcal{Z}^p( U , n) \to \mathcal{Z}^p(U , n-1).
\]
An element of $\mathcal{Z}^p(U, n)$ is called a \textit{precycle} on $U$. 
The \textit{higher Chow groups} are defined by taking the homology of this complex:
\[ CH^p(U, n) \coloneqq  H_n(\mathcal{Z}^p(U, \bullet)).\]
Note that $CH^p(U) = CH^p(U, 0)$ by the definition.

When $U$ is smooth, there is another expression via the Gersten-Milnor resolution for $CH^p(U,1)$:
\[
CH^p(U,1) \cong H(\bigoplus_{cd_xZ = p-2}K^M_2(\mathbb{C}(Z)) \to \bigoplus_{cd_xZ = p-1}K^M_1(\mathbb{C}(Z)) \to \bigoplus_{cd_xZ = p}K^M_0(\mathbb{C}(Z))).
\]
Here, $K^M_p(k)$ is the $p$-th Milnor $K$-theory of a field $k$.
Since $K^M_0(\mathbb{C}(Z)) \cong \mathbb{Z}$ and $K^M_1(\mathbb{C}(Z)) = \mathbb{C}(Z)^*$, each element of $CH^p(U,1)$ can be represented by a formal sum $\sum(f_i, Z_i)$ with a codimension $(p-1)$ subvariety $Z_i$ and a rational function $f_i$ over $Z_i$ such that $\sum_i {\rm div}(f_i) = 0$.
Taking the quotient by the image of the Tame symbols, we obtain $CH^p(U,1)$.
More specifically, the graph of $f_i|_{Z_i \setminus f_i^{-1}(1)}$ as a subvariety of $U \times (\mathbb{P}^1\setminus \{1\})$ defines an element of $\mathcal{Z}^p(U),1)$.
\begin{notation}
We consider only non-torsion higher cycles in this paper. For this reason, \textbf{we use the notation $CH^p(U, n)$ for the rational coefficient higher Chow groups $CH^p(U, n)\otimes \mathbb{Q}$ from now on}.
\end{notation}

Let $X$ be a projective variety. For a subring $\mathbb{A} \subset \mathbb{R}$, the Deligne complex is defined by a complex of sheaves on $X$
\[\mathbb{A}_{\mathcal{D}}(p) \colon \mathbb{A}(p) \to \mathcal{O}_X \to \Omega^1_X \to \ldots \to \Omega^{p-1}_X.\]
Here $\mathbb{A}(p) \coloneqq \mathbb{A}(2\pi \sqrt{i})^p$. Then the \textit{Deligne cohomology} is defined by the hypercohomology
\[H^i_{\mathcal{D}}(X, \mathbb{A}(p)) \coloneqq \mathbb{H}^i(\mathbb{A}_{\mathcal{D}}(p))\]
and Bloch defined a cycle class map
\[cl^{p, n}_{\mathcal{D}} \colon CH^p(X, n) \to H^{2p-n}_{\mathcal{D}}(X, \mathbb{Q}(p)).\]

In the case of $n=0$, $cl^{p}_{\mathcal{D}} \coloneqq cl^{p, 0}_{\mathcal{D}}$ can be considered as the unified map of the usual cycle class map $cl^p$ to the Hodge class $\Hdg ^p(X)$ and the Abel-Jacobi map $AJ^p$ to the intermediate Jacobian $J^p(X)$. More precisely, there is a commutative diagram
\[
\xymatrix{
0 \ar[r] & CH^p_{{\rm hom}}(X) \ar[r]\ar[d]^{AJ^p} & CH^p(X) \ar[r]\ar[d]^{cl^{p}_{\mathcal{D}}} & CH^p(X)/CH^p_{{\rm hom}}(X, n) \ar[d]^{cl^p}\ar[r] & 0\\
0 \ar[r] & J^p(X) \ar[r] & H^{2p}_{\mathcal{D}}(X, \mathbb{Q}(p)) \ar[r] & \Hdg^p(X) \ar[r] & 0
}
\]
with exact rows.

For a quasi-projective variety $U$, we can define the higher cycle class map to the generalized Hodge class
\begin{align*}
cl^{p, n} \colon CH^p(U, n) \to \Hdg^{p,n}(U) \coloneqq & \Hdg (H^{2p-n}(U, \mathbb{Q})(p))\\
\coloneqq & \Hom_{\MHS}(\mathbb{Q}, H^{2p-n}(U, \mathbb{Q})(p))
\end{align*}
and the higher Abel-Jacobi map from $CH^p_{{\rm hom}}(U, n) (\coloneqq \Ker (cl^{p, n}))$ to the generalized intermediate Jacobian 
\[AJ^{p,n} \colon CH^p_{{\rm hom}}(U, n) \to \Ext_{\MHS}^1(\mathbb{Q}, H^{2p-n-1}(U, \mathbb{Q})(p)).\]
 By replacing the Deligne cohomology to the absolute Hodge cohomology (\cite{KL07}, Section 2), we also can define the cycle class map $cl^{p,n}_{\mathcal{H}}$ and obtain the generalization of the above commutative diagram:
\[
\xymatrix{
0 \ar[r] & CH^p_{{\rm hom}}(U,n) \ar[r]\ar[d]^{AJ^{p,n}} & CH^p(U,n) \ar[r]\ar[d]^{cl^{p,n}_{\mathcal{H}}} & CH^p(U,n)/CH^p_{{\rm hom}}(U, n) \ar[d]^{cl^{p,n}}\ar[r] & 0\\
0 \ar[r] & J^{p,n}(U) \ar[r] & H^{2p}_{\mathcal{H}}(U, \mathbb{Q}(p)) \ar[r] & \Hdg^{p,n}(U) \ar[r] & 0.
}
\]
The composition $CH^p(U,n) \to CH^p(U,n)/CH^p_{{\rm hom}}(U, n) \xrightarrow{cl^{p,n}} \Hdg^{p,n}(U)$ is also often denoted by just $cl^{p,n}$.
For a projective $X$, however, each cohomology class has the pure Hodge structure and hence $H^{2p-n}(X, \mathbb{Z})(p)$ has no weight 0 graded pieces up to torsion. Thus $\Hdg^{p,n}(X) = \{ 0 \}$ and the diagram turns into
\[
\xymatrix{
CH^p_{{\rm hom}}(X,n) \ar@{=}[r]\ar[d]^{AJ^{p,n}} & CH^p(X,n) \ar[d]^{cl^{p, n}_{\mathcal{D}}}\\
J^{p,n}(X) \ar@{=}[r] & H^{2p-n}_{\mathcal{D}}(X, \mathbb{Q}(p)).}
\]
The vanishing of the generalized Hodge class clearly shows that we cannot state the Hodge conjecture for the higher case as the surjectivity of $cl^{p, n}$.
Instead, we consider the composition of the natural surjection $H^{2p-n}_{\mathcal{D}}(X, \mathbb{Q}(p)) \to H^{2p-n}_{\mathcal{D}}(X, \mathbb{R}(p))$ after $cl^{p, n}_{\mathcal{D}}$, which is called the \textit{real regulator map}
\[r_\mathcal{D}^{p,n} \colon CH^p(X,n) \to H^{2p-n}_{\mathcal{D}}(X, \mathbb{R}(p)).\]
Then a version of Beilinson's Hodge-$\mathcal{D}$-Conjecture is
\begin{conj}[Hodge-$\mathcal{D}$-Conjecture]
For a smooth projective variety $X$ over $\bar{\mathbb{Q}}$,
\[r_{\mathcal{D}, \mathbb{R}}^{p,n} \coloneqq r_\mathcal{D}^{p,n}\otimes \mathbb{R} \colon CH^p(X,n)\otimes \mathbb{R} \to H^{2p-n}_{\mathcal{D}}(X, \mathbb{R}(p))\]
is surjective.
\end{conj}

Note that $H_{\mathcal{D}}^{2p-1}(X, \mathbb{R}(p)) \cong H^{p-1, p-1}_{\mathbb{R}}(X)(p-1) \coloneqq H^{p-1,p-1}(X, \mathbb{R}) \otimes \mathbb{R}(p-1)$ (\cite{CL05}, Section 3).

\begin{remark}
	The same statement for quasi-projective varieties over $\mathbb{C}$ is known to be false. See \cite{MS97}.
\end{remark}

The higher Chow groups have a product structure
\[CH^p(X, n) \otimes CH^q(X, m) \to CH^{p+q}(X, n+m)\]
which is compatible with the cup products in the Deligne cohomology and the real regulator map.
Since it is known that $CH^1(X,1) \cong H^1_{\mathcal{D}}(X,\mathbb{Q}(1)) \cong  \mathbb{C}^*$, especially we obtain a map
\begin{equation}
\label{dec}
\mathbb{C}^* \otimes CH^{p-1} \to CH^p(X,1).
\end{equation}
The image $CH^p_{\dec}(X, 1)$ of the above map (\ref{dec}) is called the \textit{subgroup of the decomposable cycles} and
the \textit{group of indecomposable cycles} is defined by the quotient
\[CH^p_{\ind}(X, 1) \coloneqq CH^p(X,1) / CH^p_{\dec}(X, 1).\]
If especially the real regulator image $r_{\mathcal{D}, \mathbb{R}}^{p,1}(\gamma)$ of an element $\gamma \in CH^p(X,1)$ is not in the image ${ \rm Im} (H^1_{\mathcal{D}}(X, \mathbb{R}(1)) \otimes H^{2p-2}_{\mathcal{D}}(X, \mathbb{R}(p-1)) \xrightarrow{\mu} H^{2p-1}_{\mathcal{D}}(X, \mathbb{R}(p))) \cong \mathbb{R} \otimes \Hdg^{p-1}(X)$, we say that $\gamma$ is \textit{$\mathbb{R}$-regulator indecomposable}. Clearly $\mathbb{R}$-regulator indecomposable cycles are indecomposable.

The Deligne cohomology $H^{2p-n}_{\mathcal{D}}(X, \mathbb{A}(p))$ can be also defined as the $(-r)$-th cohomology of the Deligne cohomology complex 
\[\mathcal{M}^{\bullet} \coloneqq Cone\{\mathcal{C}_X^{2p + \bullet}(X, \mathbb{A}(p)) \oplus F^p\mathcal{D}^{2p + \bullet}_X(X) \xrightarrow{\epsilon - l} \mathcal{D}^{2p + \bullet}_X(X) \}[-1]\]
with the sheaves of topological chains and distributions on $X$.
Here, $\epsilon$ maps to the associated current and $l$ is the natural embedding.
On the other hand, the complex of precycles $\mathcal{Z}^p(X, \bullet )$ has a subcomplex $\mathcal{Z}^p_{\mathbb{R}}(X,\bullet )$ of cycles meeting real faces properly such that the inclusion is a (rational) quasi-isomorphism. The \textit{KLM-formula} \cite{KLMS06} is  a map of complexes $\mathcal{Z}^p(X, -\bullet ) \to \mathcal{M}^{\bullet}$ defined by 
\[Z \to (2 \pi i)^{p-n}((2 \pi i)^{n}T_{Z}, \Omega_Z, R_Z),\]
and inducing $AJ^{p,n}$.
 Here, each of $T_Z, \Omega_Z, R_Z $ is essentially defined by the pushforward-pull back image of the following current on $ \square^r \coloneqq  ( \mathbb{P}^1 \setminus \{ 1 \})^r$ respectively:
	\[T_r \coloneqq  (2 \pi i)^r \delta_{[- \infty, 0]^r}\]
	\[\Omega_r \coloneqq  \int_{\square^r} \wedge^r_{k=1}d \log z_k\]
\begin{multline*}
	R_r \coloneqq  \int_{\square^r} \log z_1 \wedge^r_{k=2} d \log z_k - (2 \pi i) \int_{[- \infty, 0] \times \square^{r-1}} \log z_2 \wedge_{k=3}^r d \log z_k \\
		 + \ldots + (- 2\pi i)^r \int_{[- \infty, 0]^{r-1} \times \square^1} d \log z_r.
\end{multline*}
When $T_Z = \partial\Gamma$ and $\Omega_Z = d\Xi$, by adding the differential $D((2 \pi i)^n \Gamma, \Xi, 0 ) = (-(2 \pi i )^nT_Z, - \Omega_Z, -\Xi + (2 \pi i)^n \delta_{\Gamma})$ we can simplify the formula.
Especially when $d \coloneqq \mbox{dim} X \le p$ or $p \le n$, since $F^pD^{2p-n}(X)$ vanishes and hence $\Omega_Z$ is trivial, we obtain
\begin{align*}
	AJ^{p,n}(Z)(\omega) &= (-2 \pi i )^{p-n}(R_Z+ (2\pi i )^n \delta_{\Gamma})(\omega) \\
&= \frac{1}{(- 2\pi i)^{n-p}} \left( \int_X R_Z \wedge \omega + (2\pi i )^n\int_{\Gamma}\omega \right)
\end{align*}
for each closed test form $\omega$ in $F^{d -p+1}\Omega^{2d-2p+n+1}(X)$, yielding a class in $J^{p,n}(X) \cong \{F^{d-p+1}H^{2d-2p + n +1}(X, \mathbb{C})\}^{\vee} / H_{2d-2p+n+1}(X, \mathbb{Q}(p))$.

More generally, the KLM formula holds for a smooth quasi-projective $U$, and even for a normal crossing divisor $Y$ on $X$ by changing each complex appearing in the formula to the simple complex associated to a certain double complex (See Section \ref{Higher Chow Cycles with Non-trivial Limits}). Especially it defines the cycle map
\[cl^{p-1,n-1}_{\mathcal{D}} \colon CH^{p-1}(Y,n-1) \to H^{2p-n+1}_{\mathcal{D},Y}(X, \mathbb{Q}(p))\]
and hence
\[cl^{p-1,n-1} \colon CH^{p-1}(Y,n-1) \to \Hdg (H^{2p-n+1}_{Y}(X, \mathbb{Q}(p)))\]
 for the cohomologies with support on $Y$. For the detail of the construction, see Section 5.9 of \cite{KLMS06} and Section 3 of \cite{KL07}.


\section{Construction of Families of Higher Cycles}
\label{Construction of Families of Higher Cycles}
In this section, we consider a certain family of degree $d$ surfaces $\mathcal{X}$ in $\mathbb{P}^3$ of a general form. 
Over this type of family, we can construct a family of higher cycles in $CH^2(X_t,1)$  for the general fiber $X_t$.
We classify these elements into families of decomposable cycles $\mathscr{D}$ and other two types of families of higher cycles $\mathscr{I}_0, \mathscr{I}_{\infty}$, which are $\mathbb{R}$-regulator indecomposable when $d = 4$.
In the case of quartic surfaces, in later sections we will prove the Hodge-$\mathcal{D}$-conjecture for a general fiber $X_t$ of $\mathcal{X}$ by showing that the images of $\mathscr{I}_0$ and $\mathscr{I}_{\infty}$ by $r_{\mathcal{D},\mathbb{R}{}}^{2,1}$ span the regulator indecomposable cycles $\Coker (\mu) \cong H^{1,1}_{tr}(X, \mathbb{R}(1))$.

Let  $L_i$ $(1 \leq i \leq d)$ and $M_l$ $(1 \leq l \leq d)$ be linear forms in $\mathbb{P}^3$ in general position.

We use the same notations $L_i$ and $M_l$ for the hyperplane $V(L_i)$ and $V(M_l)$ in $\mathbb{P}^3$ defined by the zero locus of these linear forms respectively.
Define a flat family of degree $d$ surfaces $\mathcal{X}$ over $\mathbb{P}^1$ by
\[ \mathcal{X} \coloneqq \{X_t \colon L_1L_2 \cdots L_d + tM_1M_2 \cdots M_d = 0 \}_{t \in \mathbb{P}^1} \subset \mathbb{P}^3 \times \mathbb{P}^1.\]
Its general fiber $X_t$ is smooth, and $X_0 = (L_1L_2 \cdots L_d = 0)$ is a simple normal crossing divisor on $\mathcal{X}$.
In fact, each point of $X_0$ has an analytic neighborhood with coordinates such that $X_t$ is defined by the equation $xy + tz = 0$ or simpler (the same holds for $X_{\infty}$).
Near $X_0$, this local equation also shows that the singular loci of the total family $\mathcal{X}$ are given by $d\binom{d}{2}$ nodes defined by
$p^{ij}_l \coloneqq L_i \cap L_j \cap M_l$.
We denote the projection  to the parameter $t$ by $\pi \colon \mathcal{X} \to \mathbb{P}^1$ and also define
$S \coloneqq \mathbb{P}^1 \setminus ({\rm discriminant\  locus})$ and $\mathcal{X}^* \coloneqq \pi^{-1}(S)$.
We write $\mathcal{L}_i, \mathcal{M}_l \subset \mathcal{X}$ for the constant families of planes defined by $L_i$ and $M_l$ respectively.
The base locus $B$ of this family $\mathcal{X}$ is obtained by
\[B = \bigcup_{1 \le i \le d}{B_i} \qquad　(B_i \coloneqq \bigcup_{1 \le l \le d} L_i \cap M_l),\]
hence $B$ includes no singular points on $X_t$.

We start by constructing some decomposable cycles.
Recall that each element of $CH^2(X,1)$ can be represented by a formal sum of pairs of divisors and rational functions over them such that the sum of their zeros and poles vanishes.
Take a constant family of lines $\mathcal{L}_i \cap \mathcal{M}_l$ as a divisor of $\mathcal{X}^*$.
Since $\mathcal{X}^*$ does not include either $X_0$ or $X_{\infty}$, the projection $\pi$ is an invertible function over $\mathcal{X}^*$.
Hence its restriction $\pi|_{\mathcal{L}_i \cap \mathcal{M}_l}$ defines an element of $\mathbb{C}^*$ via the identification $\mathcal{O}^*_{\mathbb{P}^1}(\mathbb{P}^1) \cong \mathbb{C}^*$.
Thus the pair $(\pi|_{\mathcal{L}_i \cap \mathcal{M}_l} , \mathcal{L}_i \cap \mathcal{M}_l)$ defines a family $\lambda_{il} \in CH^2(\mathcal{X}^*, 1)$ of decomposable cycles via the map (\ref{dec}). We define 
\[\mathscr{D} \coloneqq \{\lambda_{il} \mid 1 \le i, l \le d\}.\]

Next we define $\mathscr{I}_0$. Take three planes $L_i, L_j, L_k (1 \le i < j < k \le d)$ and another one $M_l$. For  each $\alpha \in \{i, j, k\}$, again we take the divisor $\mathcal{L}_{\alpha} \cap \mathcal{M}_l$, but for the rational function we take an isomorphism $\phi_{\alpha l} \colon L_{\alpha} \cap M_l \xrightarrow{\cong} \mathbb{P}^1$ on each $t \in \mathbb{P}^1$
	defined by
	\[
	\phi_{\alpha l} = \frac{L_{\sigma (\alpha)}}{L_{\sigma^2 (\alpha)}}\big|_{M_l}.
	\]
Here, $\sigma \in \mathfrak{S}_3$ is the cyclic permutation defined by
$\begin{pmatrix}
i & j & k \\
j & k & i
\end{pmatrix}
$.
Hence $\phi_{\alpha l}^{-1}(0) = L_{\alpha} \cap L_{\sigma (\alpha)} \cap M_{l}$ and $\phi_{\alpha l}^{-1}(\infty) = L_{\alpha} \cap L_{\sigma^2 (\alpha)} \cap M_{l}$.
We use the same notation $\phi_{\alpha l}$ for the rational function over $\mathcal{X}$ defined by $\phi_{\alpha l}$ constantly with respect to $t$.
Then we obtain a precycle
\[\Gamma_{\alpha l} \coloneqq (\phi_{\alpha l}, \mathcal{L}_{\alpha} \cap \mathcal{M}_l) \in \mathcal{Z}^2(\mathcal{X}^*, 1).\]
By the definition of $\phi_{\alpha l}$, $\partial (\Gamma_{\alpha l})$ is the divisor $ [\mathcal{L}_i \cap \mathcal{L}_{\sigma (\alpha)}] - [\mathcal{L}_i \cap \mathcal{L}_{\sigma^2 (\alpha)}]$.
Hence the precycle
\[\gamma_{ijk,l} \coloneqq \Gamma_{i l} + \Gamma_{j l} + \Gamma_{k l}\]
satisfies $\partial(\gamma_{ijk,l}) = 0$.
Thus we obtain a higher cycle $\gamma_{ijk,l} \in CH^2(\mathcal{X}^*, 1)$.
We use the same notation $\gamma_{ijk,l} \in CH^2(X_t, 1)$ for each fiber at $t=0$ (Figure 1).
We define
\[\mathscr{I}_0 \coloneqq \{\gamma_{ijk, l} \mid 1 \le i < j < k \le d, 1 \le l \le d \}.\]

Finally, $\mathscr{I}_{\infty}$ is defined by changing $L$ and $M$ in the above construction of $\mathscr{I}_0$.
Specifically, for three planes $M_l, M_m, M_n$ and $L_i$ and for each $\beta \in \{l, m, n\}$,
 we take an isomorphism $\psi_{i\beta} \colon L_i \cap M_{\beta} \xrightarrow{\cong} \mathbb{P}^1$
 such that $\psi_{i \beta}^{-1}(0) = L_i \cap M_{\beta} \cap M_{\sigma (\beta)}$ and $\psi_{i \beta}^{-1}(\infty) = L_i \cap M_{\beta} \cap  M_{\sigma^2 (\beta)}$.
 Then it defines a precycle $\Gamma'_{i\beta} \coloneqq (\psi_{i\beta}, \mathcal{L}_i \cap \mathcal{M}_{\beta})$ and we can see that
 \[\delta_{i, lmn} \coloneqq \Gamma'_{il} + \Gamma'_{im} + \Gamma'_{in}\]
is also an element of $CH^2(\mathcal{X}^*, 1)$.
We define
\[\mathscr{I}_{\infty} \coloneqq \{\delta_{i, lmn} \mid 1 \le i \le d, 1 \le l < m < n \le d \}.\]

\begin{figure}[H]
  \centering
  \includegraphics[width=10cm]{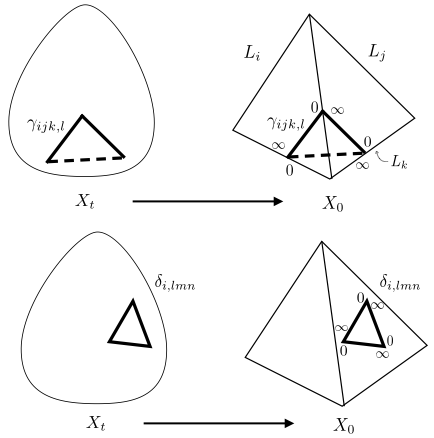}
    \caption{Higher Chow cycle $\gamma_{ijk,l}$ and $\delta_{i,lmn}$}
\end{figure}


\section{Singularities and limits of Normal Functions
\label{Singularities and limits of Normal Functions}}
In this section, we review two invariants of normal functions which are called the singularity and limit. When we obtain a family of higher cycles over a family of projective varieties, the Abel-Jacobi map defines the corresponding admissible normal function.
When the family is a semistable degeneration we can observe its singularity by using the Clemens-Schmid exact sequence. If the normal function has the trivial singularity, then we obtain its limit value in the limiting Jacobian.

Let $\overline{S}$ be a complex manifold and $j \colon S \to \overline{S}$ be an open immersion of a Zariski open subset $S$.
For a variation of Hodge structure $\mathcal{H}$, its generalized Jacobian bundle is defined by
\[J(\mathcal{H}) \coloneqq \frac{\mathcal{H}}{\mathcal{F}^0\mathcal{H} + \mathbb{H}_{\mathbb{Q}}}.\]
A holomorphic horizontal section of $J(\mathcal{H})$ is called a \textit{$J(\mathcal{H})$-valued normal function} over $S$.
The group of $J(\mathcal{H})$-valued normal functions $NF(S, \mathcal{H})$ is canonically isomorphic to 
$\Ext^1_{\VMHS (S)}(\mathbb{Z}, \mathcal{H})$ with the category $\VMHS (S)$ of variations of mixed Hodge structures over $S$. Moreover, $\VMHS (S)$ contains the subcategory $\VMHS (S)^{\rm ad}_{\overline{S}}$ of admissible variations of Hodge structures (\cite{Sa96}).
An element of the subgroup $\Ext ^1_{\VMHS (S)^{\rm ad}_{\overline{S}}}(\mathbb{Z}, \mathcal{H})$ of $\Ext ^1_{\VMHS (S)}(\mathbb{Z}, \mathcal{H})$ is called an \textit{admissible normal function} with respect to $\overline{S}$.
By Section 2 of \cite{BFNP09}, the group of admissible normal functions $NF(S, \mathcal{H})^{\rm ad}_{\overline{S}}\otimes \mathbb{Q}$ with rational coefficients is isomorphic to $\Ext ^1_{\MHM (S)^{ps}_{\overline{S}}}(\mathbb{Q}, \mathcal{H})$.
Here, $\MHM (S)^{ps}_{\overline{S}}$  is the category of smooth polarizable mixed Hodge modules over $S$.

Let $\nu$ be an admissible normal function over $S$ and $\iota_s$ be the embedding of a point $s \in \overline{S}$.
We define a map $sing_s$ by the composition 
\begin{align*}
NF(S, \mathcal{H})^{\rm ad}_{\overline{S}}\otimes \mathbb{Q} \cong 
&\Ext ^1_{\MHM (S)^{ps}_{\overline{S}}}(\mathbb{Q}, \mathcal{H})\\
\xrightarrow{(\iota_s^*Rj_*)^{\Hdg}} 
&\Ext ^1_{D^b\MHM (\{s\})}(\mathbb{Q}, \iota_s^*Rj_*\mathcal{H})\\
\cong
&\Ext ^1_{D^b\MHS}(\mathbb{Q}, \iota_s^*Rj_*\mathcal{H}) \\
\to
&\Hom_{\MHS}(\mathbb{Q}, H^1((\iota_s^*Rj_*)\mathcal{H})).
\end{align*}
The invariant $sing_s(\nu)$ is called the \textit{singularity of the normal function} $\nu$ at $s$.
From the spectral sequence for the cohomology functor and $\Hom_{\MHS}(\mathbb{Q}, -)$, we also obtain a natural map $lim_s \colon \Ker (sing_s) \to \Ext^1_{\MHS}(\mathbb{Q}, H^0(\iota_s^*Rj_*\mathcal{H}))$ which makes the following commutative diagram:
\[
\xymatrix@M=7pt{
& 0\\
 & \Hom_{\MHS}(\mathbb{Q}, H^1((\iota_s^*Rj_*)\mathcal{H}))\ar[u]\\
 NF(S, \mathcal{H})^{\rm ad}_{\overline{S}}\otimes \mathbb{Q} \ar[ru]^-{sing_s}\ar[r]^-{(\iota_s^*Rj_*)^{\Hdg}} & \Ext ^1_{D^b\MHS}(\mathbb{Q}, \iota_s^*Rj_*\mathcal{H})\ar[u] \\
 \Ker (sing_s) \ar[u]\ar[r]^-{lim_s}& \Ext^1_{\MHS}(\mathbb{Q}, H^0(\iota_s^*Rj_*\mathcal{H})) \ar[u]\\
 & 0 \ar[u]
}
\]

We can apply the above theory of admissible normal functions to a family of higher cycles on smooth projective varieties, because of the following result of Brylinski and Zucker:
Let $f \colon \mathfrak{X}^* \to S$ be a smooth proper family of quasi-projective varieties.
A higher cycle
\[\mathfrak{Z}^* \in CH^p(\mathfrak{X}^*,n)_{prim} \coloneqq \bigcap_{x \in S}\Ker (CH^p(\mathfrak{X}^*, n) \to CH^p(X_x, n) \to \Hdg^{p,n}(X_x))\]
 defines a holomorphic section $\nu_{\mathfrak{Z^*}}$ of $J(\mathcal{H}^{p,n})$ for $\mathcal{H}^{p,n} \coloneqq R^{2p-n-1}f_*\mathbb{Q}(p) \otimes \mathcal{O}_S$ by taking the fiberwise Abel-Jacobi values.
\begin{thm}\cite{BZ90}
\label{Zucker}
$\nu_{\mathfrak{Z}^*}$ is an admissible normal function.
\end{thm}

If $\mathfrak{X}^*$ is the restriction of a proper family $\mathfrak{X}$ over $\overline{S}$ to $S$ and $\mathfrak{Z}^*$ is that of a family of higher cycle $\mathfrak{Z} \in CH^p(\mathfrak{X}, n)$,
with the complement $\mathfrak{X}_{sing} \coloneqq \mathfrak{X} \setminus \mathfrak{X}^*$, we obtain the localization exact sequence
\[\cdots \to CH^p(\mathfrak{X}_{sing}, n) \to CH^p(\mathfrak{X}, n) \to CH^p(\mathfrak{X}^*, n) \xrightarrow{res} CH^{p-1}(\mathfrak{X}_{sing}, n-1) \to \cdots.\]
Here, the morphism $res$ is defined by Bloch's moving lemma. In fact, for each $\gamma \in CH^p(\mathfrak{X}^*, n)$ this lemma guarantees that there exists a precycle $\Gamma \in \mathcal{Z}^p(\mathfrak{X}, n)$ such that its restriction to $\mathfrak{X}^*$ is a higher cycle with the same class to $\gamma$. We can see that $res(\gamma) \coloneqq \partial{\Gamma}$ is actually in $\mathfrak{X}_{sing}$.

Now we consider the special case that $\mathfrak{X}$ is a one-parameter semistable degeneration of $d$ dimensional projective varieties.
It means that $\overline{S}$ is a projective curve and each singular fiber $X_{s_0} \subset \mathfrak{X}_{sing}$ is a reduced simple normal crossing divisor.
Take a point in the discriminant locus $s_0 \in \overline{S} \setminus S$ and let $\Delta \subset \overline{S}$ be the unit disk in a local coordinate of $\overline{S}$ with the origin $s_0$.
By changing the coordinate of $S$ if we need, we may assume that $X_{s_0}$ is the unique singular fiber in  the restriction $\mathfrak{X}|_{\Delta}$.

The upper half plane $\mathfrak{H}$ can be considered the universal cover of $\Delta^* \coloneqq \Delta \setminus \{ 0 \}$.
With the base change $\mathfrak{X}_{\mathfrak{H}} \coloneqq \mathfrak{X}|_{\Delta} \times_{\Delta^*} \mathfrak{H}$, we obtain the commutative specialization diagram
\[
\xymatrix{
\mathfrak{X}_{\mathfrak{H}} \ar[r]^k\ar[d]  & \mathfrak{X}|_{\Delta} \ar[d]  &  X _{s_0}\ar[l]_i\ar[d]\\
\mathfrak{H}\ar[r]  &  \Delta  &  \{0\}.\ar[l]
}
\]
Since $\mathfrak{X}_{\mathfrak{H}}$ is homotopic to any general fiber $X_t$ $(t \neq 0)$, we can define the specialization map
\[sp \colon H^k(X_{s_0},  \mathbb{Q}) \to H^k(\mathfrak{X}_{\mathfrak{H}},  \mathbb{Q}) \cong  H^k(X_t, \mathbb{Q})\]
induced by the adjoint morphism $\mathbb{Q}_{X_{s_0}} \to  i^*Rk_*k^*\mathbb{Q}_{\mathfrak{X}|_{\Delta}}$.
We remark that originally this map is defined analytically by Clemens' retraction $\mathfrak{X}|_{\Delta} \to X_{s_0}$ , but generally this retraction is not holomorphic (\cite{Cl77}.
Also, since the local monodromy $T$ around $s_0$ is unipotent, it defines the log monodromy action 
\[N_{s_0} \coloneqq \sum_{l =1}\frac{(-1)^{l-1}}{l}(T-I)^l \colon H^k(X_{t},  \mathbb{Q}) \to H^k(X_{t},  \mathbb{Q}).\]

Generally, for any given nilpotent operator $N \colon H \to H$ on a finite dimensional vector space $H$ and a fixed integer $k$, there exists a unique increasing filtration $W_{\bullet}$ on $H$ such that $N(W_i) \subset W_{i-2}$ and $N^l \colon \Gr^W_{k+l}H \to \Gr^W_{k-l}H$ is an isomorphism for arbitrary $l \geq 0$. This is called the weight filtration of $N$ centered at the integer $k$.

On $H^k(X_{t},  \mathbb{C})$, the monodromy weight filtration (at $t = s_0$) is defined by the weight filtration of the log monodromy action $N_{s_0}$ centered at $k$. Steenbrink exhibited this filtration in terms of the usual weight filtration on a double complex $A^{p,q} = \Omega^{p+q+1}_{X_t}(\log X_{s_0})/W_p\Omega^{p+q+1}_{X_t}(\log X_{s_0})$ consists of the log de Rham complex and extended it to the rational structure $H^k(X_{t},  \mathbb{Q})$. A decreasing filtration $F_{\bullet}$ on $ H^k(X_{t},  \mathbb{C})$ can be also defined from $A^{\bullet, \bullet}$, and the monodromy wight filtration $W_{\bullet}$ and $F_{\bullet}$ define
 a mixed Hodge structure on $H^k(X_t,  \mathbb{Q})$ such that $sp$ and $N_{s_0}$ are morphisms of mixed Hodge structures of weight $0$ and $-1$ respectively (with the usual mixed Hodge structure on $H^k(X_{s_0},  \mathbb{Q})$). See \cite{Cl77} and Chapter 11 of \cite{PS08}.
This is called the limiting mixed Hodge structure (LMHS) and we denote $H^k_{lim}(X_t,  \mathbb{Q})$ for $H^k(X_t,  \mathbb{Q})$ equipped with LMHS.
With this mixed Hodge structure, we obtain the Clemens-Schmid exact sequence:
\begin{thm}\cite{Cl77}
\label{CS exact seq}
There is a long exact sequence of mixed Hodge structures
\begin{align*}
\cdots \to H^k(X_0, \mathbb{Q}) \xrightarrow{sp} H^k_{lim}(& X_t, \mathbb{Q}) \xrightarrow{N_{s_0}} H^k_{lim}(X_t, \mathbb{Q}(-1))\\
 &\xrightarrow{\alpha} H_{2d-k}(X_{s_0}, \mathbb{Q}(-d-1)) \xrightarrow{\phi} H^{k+2}(X_{s_0}, \mathbb{Q}) \to \cdots.
\end{align*}
 Here, $\phi$ is the composition of the Poincar\'e-Lefschetz duality
\[ H_{2d-k}(X_{s_0}, \mathbb{Q}(-d-1)) \cong H^{k+2}_{X_{s_0}}(\mathfrak{X}|_{\Delta}; \mathbb{Q}) \coloneqq H^{k+2}(\mathfrak{X}|_{\Delta}, \mathfrak{X}|_{\Delta^*}; \mathbb{Q}),\] the natural morphism $H^{k+2}(\mathfrak{X}|_{\Delta}, \mathfrak{X}|_{\Delta^*}; \mathbb{Q}) \to H^{k+2}(\mathfrak{X}|_{\Delta}, \mathbb{Q})$ and the isomorphism $H^{k+2}(\mathfrak{X}|_{\Delta}, \mathbb{Q}) \cong H^{k+2}(X_{s_0}, \mathbb{Q})$.
Moreover $\alpha$ factors through $H^{k+1}(\mathfrak{X}|_{\Delta^*},\mathbb{Q})$.

\end{thm}

Since $\dim S = 1$, $\iota_s^*Rj_*\mathcal{H}^{p,n}$ is quasi-isomorphic to the complex 
\[\{H^{2p-n-1}_{lim}(X_t, \mathbb{Q}(p)) \xrightarrow{N_s} H^{2p-n-1}_{lim}(X_t, \mathbb{Q}(p-1)) \}\]
(\cite{KP11}). Hence in this case,  we may consider the singularities and limits as the invariants in
\[ \Hdg(\Coker N_{s_0}) \coloneqq \Hom_{\MHS}(\mathbb{Q}, \Coker N_{s_0})\]
and
\[ J_{lim,{s_0}} \coloneqq \Ext^1_{\MHS}(\mathbb{Q}, \Ker(N_{s_0})) \cong J(H^{2p-n-1}_{lim}(X_t, \mathbb{Q}(p)))\]
respectively. As an extension class, we can represent the admissible normal function $\nu$ by a short exact sequence
\[0 \to \mathcal{V} \to \mathcal{E}_{\nu} \to \mathbb{Q}_{S} \to 0\]
of variations of mixed Hodge structures with the underlying local systems
\[0 \to \mathbb{V} \to \mathbb{E}_{\nu} \to \mathbb{Q}_{S} \to 0.\]
 Deligne's extension $\tilde{\mathbb{E}}_{\nu} \coloneqq e^{-\frac{1}{2 \pi i }\log t N_t}\mathbb{E}_{\nu}$ defines the extension $\mathcal{E}_{\nu,e} \coloneqq \tilde{\mathbb{E}}_{\nu} \otimes \mathcal{O}_{\Delta} $ and the  admissibility of $\nu$ means that $\nu = \nu_F - \nu_{\mathbb{Q}}$ in the Jacobian bundle with a lift $\nu_F$ and $\nu_{\mathbb{Q}}$ of 1 to $\mathcal{E}_{\nu,e}$ and $\tilde{\mathbb{E}}_{\nu, 0}$ respectively such that $\nu_F|_{\Delta^*}$ is in the Hodge filtration $\mathcal{F}^0(\mathcal{E}_{\nu})$ and $N\nu_{\mathbb{Q}}$ is in the monodromy weight filtration $W_{-2}\tilde{\mathbb{V}}_0$. With these notations, specifically the singularity at $s_0 = 0 \in \Delta $ is given by
\[sing_{s_0}(\nu) = [N\nu_{\mathbb{Q}}] (\equiv [N\nu_F(0)]).\]

Let $\mathfrak{Z}^*$ be a higher cycle over $\mathfrak{X}^*$.
If the general fiber $X_t$ is a smooth projective variety, $\mathfrak{Z}^*$ is in $CH^p(\mathfrak{X}^*,n)_{prim}$ since the generalized Hodge classes vanish.
By Theorem \ref{Zucker}, it defines the admissible normal function $\nu_{\mathfrak{Z}^*}$ .
Hence we obtain a map $\mathcal{AJ}^{p,n} \colon CH^p(\mathfrak{X}^*, n) \to NF(S, \mathcal{H}^{p,n})^{\rm ad}_{\overline{S}}\otimes \mathbb{Q}$.
On the other hand, since $S$ is a curve, the codomain of singularities is $\Hdg (\Coker (N_{s_0})) $ for $N_{s_0} \colon H^{2p-n-1}_{lim}(X_t, \mathbb{Q}(p)) \to H^{2p-n-1}_{lim}(X_t, \mathbb{Q}(p-1))$ as we have seen above .
By Theorem \ref{CS exact seq}, this group can be regarded as a subgroup of $H_{2d-2p+n+1}(X_{s_0} , \mathbb{Q}(p-d-1))$. We denote
\[\Hdg_{p-1, n-1}(X_{s_0}) \coloneqq \Hdg (H_{2d-2p+n+1}(X_{s_0}, \mathbb{Q})).\]
Note that the Poincar\'e-Lefschetz duality isomorphism
induces the natural map
\[\beta \colon \Hdg (H^{2p-n+1}_{X_{s_0}}(\mathfrak{X}, \mathbb{Q}(p))) \to \Hdg_{p-1, n-1}(X_{s_0})\]
since the isomorphism is a morphism of mixed Hodge structures.

Since each singular fiber is a simple normal crossing divisor, we can consider the cycle map from $CH^{p-1}(X_{s_0},n-1)$ as the end of the previous section. With this map, we obtain a relation of $res(\mathfrak{Z})$ and $sing_{s_0}(\nu_{\mathfrak{Z}^*})$:

\begin{prop}
\label{cd for sing}
Suppose $\mathfrak{X} \to \overline{S}$ be a semistable degeneration of smooth projective varieties and $n \ge p$ or $p \ge d$.
Then for each $s_0 \in \overline{S}\setminus S$, there is a commutative diagram
\[
\xymatrix@M=7pt{
CH^p(\mathfrak{X^*}, n) \ar[r]^{\mathcal{AJ}^{p,n}}\ar[d]^{res}\ar[rd]^{cl^{p,n}} & NF(S, \mathcal{H}^{p,n}_{f})^{\rm ad}_{\overline{S}}\otimes \mathbb{Q}\ar[r]^{sing_{s_0}} & \Hdg (\Coker (N_{s_0})) \ar@{^(->}[dd]^{\Hdg(\alpha)}\\
CH^{p-1}(\mathfrak{X}_{sing},n-1) \ar[d]^{i_{s_0}^*}& \Hdg^{p,n} (\mathfrak{X}^*) \ar[rd]^{\Hdg (r)}\\
CH^{p-1}(X_{s_0},n-1) \ar[r]^-{cl^{p-1,n-1}} & \Hdg (H^{2p-n+1}_{X_{s_0}}(\mathfrak{X}, \mathbb{Q}(p))) \ar[r]^{\beta} & \Hdg_{p-1, n-1}(X_{s_0}).
}
\]
Here, $i_{s_0}$ is the inclusion $X_{s_0} \subset \mathfrak{X}_{sing}$ and $r$ is defined as the composition of natural maps $H^{2p-n}(\mathfrak{X}^*, \mathbb{Q}(p)) \to H^{2p-n}(\mathfrak{X}^*|_{\Delta^*}, \mathbb{Q}(p)) \xrightarrow{res_{s_0}} H^{2p-n+1}(\mathfrak{X}|_{\Delta}, \mathfrak{X}|_{\Delta} \setminus X_{s_0}; \mathbb{Q}(p))$ and the Poincar\'e-Lefschetz duality isomorphism.
\end{prop}
\begin{proof}
The commutativity of the lower triangular diagram follows from the functoriality of the cycle maps for the pull back and residue maps. To see that of the upper triangular diagram, recall that the image of the cycle map $cl^{p,n}(\mathfrak{Z}^*)$ for a given family $\mathfrak{Z}^*$ in $CH^p(\mathfrak{X}^*,n)$ is obtained by the class of a topological cycle $[(2 \pi i )^pT_{\mathfrak{Z}'^*}]$ via the KLM formula (for the complement of a normal crossing divisor) $cl^{p,n}_{\mathcal{D}}(\mathfrak{Z}^*) = [(2 \pi i )^{p-n}((2 \pi i )^n T_{\mathfrak{Z}'^*}, \Omega_{\mathfrak{Z}'^*}, R_{\mathfrak{Z}'^*})]$  with a representative $\mathfrak{Z}'^*$ in $\mathcal{Z}^p_{\mathbb{R}}(\mathfrak{X}^*, n)$ of $\mathfrak{Z}^*$.  Hence $\Hdg (r) \circ cl^{p,n} (\mathfrak{Z}^*)$ coincides with the dual of  $(2 \pi i )^pres_{s_0}([T_{\mathfrak{Z}'^*}])$. 

On the other hand, we can take a chain $\Gamma_t$ with $ \partial \Gamma_t = T_{Z_t}$ for the fiber $Z_t$ of $\mathfrak{Z}^*$ at general $t$, since $[T_{Z_t}] = 0$.
From the assumption $n \ge p$ or $p \geq d$, $\Omega_{Z_t} =0$, and hence we may simplify the triple for $cl^{p,n}_{\mathcal{D}}(Z_t)$ to $(0, 0, R'_{Z_t} \coloneqq (2 \pi i )^{p-n}R_{Z_t} + (2 \pi i )^p \delta_{\Gamma_t})$ by adding $D((2 \pi i )^p \Gamma_t, 0,0) = (0,0,(2 \pi i )^{p}\delta_{\Gamma_t})$. Therefore $\nu(t) \coloneqq AJ^{p,n}(Z_t) = \nu_{\mathbb{Q}}(t) - \nu_F(t)$ can be represented by the family of currents $\{ R'_{Z_t} \}$, on whose class $[R'_{Z_t}]$ the Gauss-Manin connection $\nabla$ is computed by locally, lifting the $\{R'_{Z_t}\}$ to $R'_{\mathfrak{Z}^*_U}$ and applying $d$ to get $\Omega_{\mathfrak{Z}^*_U}$. Hence $\nabla\nu = [\Omega_{\mathfrak{Z}^*}] = cl^{p,n}(\mathfrak{Z}^*)$.
 It is well-known that $res_{s_0}(\nabla ) = - 2\pi i N$, therefore
 \begin{align*}
 	sing_{s_0} \circ \mathcal{AJ}^{p,n}(\mathfrak{Z}^*)   = N \nu _F (0)  & = (-2 \pi i )^{-1}res_{s_0}(\nabla)(\nu _F(0)) \\
 	& = (2 \pi i)^{-1}res_{s_0}( \nabla )( \nu ) \\
 	& = (2 \pi i)^{-1} res_{s_0}(\nabla \nu) = (2\pi i )^{-1}(res_{s_0} \circ cl^{p,n})(\mathfrak{Z}^*).
  	 \end{align*}
 Since $\alpha$  is a morphism of type (1-d,1-d) it coincides with the above computation.

\end{proof}
By Theorem \ref{CS exact seq}, $\Coker(N_{s_0})$ is isomorphic to $\Ker(\phi)$.
Hence we obtain the composition 
\[\widetilde{sing}_{s_0} \coloneqq \Hdg(\beta) \circ cl^{p-1, n-1} \circ i^*_{s_0} \circ res \colon CH^p(\mathfrak{X^*}, n) \to \Hdg(\Ker (\phi)).
\]
\begin{cor}
\label{Cor for sing}
In the situation of the above proposition,
$sing_{s_0}(\nu_{\mathfrak{Z}^*}) \neq 0$ if and only if $\widetilde{sing}_{s_0}(\mathfrak{Z}^*) \neq 0$.
\end{cor}

We also can describe the limit invariant $lim_{s_0}(\nu_{\mathfrak{Z}^*})$  as follows under the assumption that $res(\mathfrak{Z}^*)$ vanishes. 
Since the specialization map $sp \colon H^{2p-n-1}(X_{s_0}, \mathbb{Q}(p)) \to H^{2p-n-1}_{lim}(X_t, \mathbb{Q}(p))$ is a morphism of MHS, it induces a map $J(sp)\colon J(X_{s_0}) \to J_{lim,s_0}$. Now $lim_{s_0}(\nu_{\mathfrak{Z}^*})$ is an invariant in the right hand side, but we also can extend $\mathfrak{Z}^*$ to a higher cycle $\mathfrak{Z}$ in $CH^p(\mathfrak{X},n)$. It defines the pullback $Z_0$ in $H_{\mathcal{M}}^{2p-n}(X_{s_0},\mathbb{Q} (p))$. Then
\begin{thm}\cite{dDI17}
\label{limit}
\[lim_{s_0}(\nu_{\mathfrak{Z}^*}) = J(sp)(AJ_{X_{s_0}}(Z_0)).\]
	
\end{thm}
Especially, when we have a family $\omega(t)$ of the class in $\Hdg(H_{2p-n-1}(X_t, \mathbb{Q}(-p)))$ such that it lifts to a class on $\mathbb{\mathfrak{X}^*}$ with non-trivial residue on $X_0$, dually it induces a splitting 
\[\eta \colon H^{2p-n-1}(X_0, \mathbb{Q}(p)) \twoheadrightarrow \mathbb{Q}(p)\]
of the morphism of MHS. The analytic limit of the paring $\langle\nu_{\mathfrak{Z}}(t), \omega(t)\rangle$ can be obtained as the period
\[\lim_{t \to s_0}\langle\nu_{\mathfrak{Z}}(t), \omega(t)\rangle \equiv J(\eta)(AJ_{X_0}(Z_0))\]
in $ J(\mathbb{Q}(p)) \cong \mathbb{C}/\mathbb{Q}(p)$.
\begin{remark}
	More generally, the above theorem does not require the SSD condition. (\cite{dDI17}, Section 5.3).
\end{remark}


\section{Higher Chow Cycles with Non-trivial Singularities}
\label{Higher Chow Cycles with Non-trivial Singularities}
We use the same notations as Section \ref{Construction of Families of Higher Cycles}.
In this section we shall prove the following statement.

\begin{thm}
\label{main thm}
Let $\ast$ be $0$ or $\infty$.
For general choices of $\{L_i\}$ and $\{M_l\}$,
$sing_{\ast} \circ \mathcal{AJ}^{2,1} (\mathscr{I_{\ast} \cup \mathscr{D}})$ spans $\Hdg(\Coker (N_{\ast})) \subset \Hdg_{1, 0}(X_{\ast})$.
\end{thm}

To apply the discussions in the previous section, first of all, we resolve the singularities of $\mathcal{X}$ to obtain a semistable degeneration family $\tilde{\mathcal{X}}$.
Recall that $\mathcal{X}$ has $d\binom{d}{2}$ nodal singularities $\{p^{ij}_l\}_{1 \le i < j \le d, 1 \le l \le d}$ which are included in the base locus $B = \bigcup B_i$.
If we denote $\mathbb{P}^3 = P^0$ and define a successive blow up $P^i$ of $\mathbb{P}^3$ inductively by the blow up $b_i \colon P^{i} \to P^{i-1}$ of $P^{i-1}$ along the strict transformation of $B_{i}$ in $P^{i-1}$, then the composition $b_i \circ b_{i-1} \circ \cdots \circ b_1$ defines a strict transformation $\mathcal{X}^i \to \mathcal{X}$. Since each $p^{ij}_l \in B_i$ is a node, this strict transformation resolves $p^{ij}_l$.
We define a smooth family $\tilde{\mathcal{X}}$ by taking a resolution of the remaining singularities in $\mathcal{X}^d$. Denote the composition of these resolutions by $b \colon \tilde{\mathcal{X}} \to \mathcal{X}$.
Though a singular fiber $\tilde{X}_{t_0}$ of  $\tilde{\mathfrak{X}}$ may not be a simple normal crossing unless $t_0 = 0$ or $\infty$, by the semistable reduction theorem (\cite{KKMS73}), we may assume that $\tilde{\pi} \colon \tilde{\mathcal{X}} \to S$ with a finite cover $S \to \mathbb{P}^1$ is a semistable degeneration family after repeating base changes and desingularizations.
Note that $\tilde{X}_t \cong X_t$ and $\tilde{X}_0 \cong (\mathcal{X}^d)_0$ is given by adding the exceptional curve $E^{ij}_l \coloneqq b^{-1}(p_{ij, l})$ to $L_j \subset X_0$.
More precisely, for the strict transformation $\tilde{L_j}$ of $L_j$, 
\[
\Pic(\tilde{L_j}) \cong b^*\Pic(L_j)\oplus (\bigoplus_{i < j, l}\mathbb{Z}[E^{ij}_l]) \cong (\mathbb{Z}l_j )\oplus (\bigoplus_{i < j, l}\mathbb{Z}e^{ij}_l).
\]
Here, $l_j$ is the divisor class of the general line in $L_j \cong \mathbb{P}^2$ and $e^{ij}_l$ is that of $E^{ij}_l \cong \mathbb{P}^1$.

Now, we have the invariant $\widetilde{sing}_{0} \colon CH^2(\tilde{\mathcal{X}}, 1) \to \Hdg(\Ker (\phi)) $ by the discussion in the previous section.
Since both $X_ {\ast}$ and $\mathscr{I}_{\ast}$ ($\ast = 0, \infty$) have the symmetry by replacing each linear form $L_i$ with $M_i$ and $M_l$ with $L_l$, we also can consider $\widetilde{sing}_{\infty}$ with another blow up $b' \colon \tilde{\mathcal{X}}' \to \mathcal{X}$ defined by replacing $B_i = \bigcup (L_i \cap M_l)$ by $B'_i \coloneqq \bigcup (M_i \cap L_l)$ in the above construction of $\tilde{\mathcal{X}}$.
From now on we consider only $\widetilde{sing}_{0}$, but one can obtain exactly the same result for $\widetilde{sing}_{\infty}$ by replacing linear forms.

\begin{lem}
\label{basis}
Consider $l_i$ and $e^{ij}_l$ as elements in $H_{2}(\tilde{X}_0, \mathbb{Q}(-1))$.
Then a basis of the $\mathbb{Q}$-vector space $\Hdg(\Ker(\phi)) \subset \Hdg(H_2(\tilde{X}_0, \mathbb{Q}(-1)))$ is given by
\[
\mathcal{B} \coloneqq \left\{ \sum_{1 \le i \le d} l_i \right\} \cup \left\{ \sum_{1 \le l' \le d}(e^{ij}_l - e^{ij}_{l'}) \right\}_{1\le i<j \le d, 1\le l \le (d-1)}
\]
In particular, 
\[{\rm dim}(\Hdg(\Coker (N_{0}))) = \Hdg(\Ker(\phi)) = 1 + (d-1)\binom{d}{2}\].
\end{lem}

\begin{proof}
We firstly find a basis of $\Hdg(H_2(\tilde{X}_0, \mathbb{Q}(-1))) = H_2(\tilde{X}_0, \mathbb{Q})^{(-1,-1)}$, and then find that of $\Hdg(\Ker(\phi))$.
For the simplicity we denote 
$Y \coloneqq \tilde{X}_0$, $Y_I \coloneqq \bigcap_{i \in I} \tilde{L}_i$ and $Y^{[k]} \coloneqq \coprod_{|I| = k+1} Y_I$.

The weight spectral sequence in this case is given by dualizing that for cohomology groups (\cite{Del74}):
\[E^1_{p,q} = H_q(Y^{[p]}, \mathbb{Q}) \Rightarrow H_{p+q}(Y, \mathbb{Q}).\] 
Since it degenerates at $E^2$ and differentials for cohomology groups are compatible with the Gysin morphisms,
\[\Gr^W_{-2}(H_2(X_0, \mathbb{Q})) = E^{\infty}_{0,2} \cong \Coker (d_1 \colon H_2(Y^{[1]},\mathbb{Q}) \to H_2(Y^{[0]}, \mathbb{Q}))\]
and the differential $d_1$ is  given by the natural morphism.

Since the strict transformation $\widetilde{L_i \cap L_j}$ is isomorphic to the original line $L_i \cap L_j$, $H_2(\widetilde{L_i \cap L_j}, \mathbb{Q})^{(-1,-1)}$ $(i < j)$ is generated by the unique class $l_{ij}$.
Each of $(l_{ij})$ induces a relation in $H_2(Y^{[0]}, \mathbb{Q})$ via $d_1 \colon H_2(Y^{[1]}, \mathbb{Q}) \to H_2(Y^{[0]}, \mathbb{Q})$.
To see this relation, we should represent $l_{ij}$ in each of $\Pic(\tilde{L}_i)$ and $\Pic(\tilde{L}_j)$ with respect to the above basis.
The intersection products for each $i < j$ are given by
\begin{align*}
&(l_i \cdot l_i)_{\tilde{L}_i} = 1\\
&(l_j \cdot e^{ij}_l)_{\tilde{L}_j} = 0\\
&(e^{ij}_l\cdot e^{i'j}_{l'})_{\tilde{L}_j} = \begin{cases}
							-1 & ((i,l) = (i',l'))\\
							0 & ((i,l) \neq (i', l')),
							\end{cases}
\end{align*}
and 
\begin{align*}
&(l_{ij} \cdot l_i)_{\tilde{L}_i} = (l_{ij} \cdot l_j)_{\tilde{L}_j} = 1\\
&(l_{ij} \cdot e^{ij}_l)_{\tilde{L}_j} = 1\\
&(l_{ij} \cdot e^{i'j}_{l})_{\tilde{L}_j} = 0 　\hspace{5mm}　(i \neq i').
\end{align*}
Hence we can see that
\begin{align*}
&l_{ij} = l_i  \quad {\rm in } \quad \Pic(\tilde{L}_i)\\
&l_{ij} = l_j - (\sum_{l} e^{ij}_l) \quad {\rm in } \quad \Pic(\tilde{L}_j).
\end{align*}
Therefore
\[
H_2(Y, \mathbb{Q})^{(-1,-1)} \cong \langle \; \{l_i\}_{1 \le i \le d},  \{e^{ij}_l\}_{1 \le j < i \le d, 1 \le l \le d} \; \rangle \; / \;  \{ l_j - l_ i = \sum_{1 \le l \le d} e^{ij}_l\}_{1 \le i < j \le d}.
\]
Since each relation is independent from others, it also shows that 
\[{\rm dim}(H_2(Y , \mathbb{Q})^{(-1,-1)})= d + d\binom{d}{2} - \binom{d}{2} = d \bigl(1 + \frac{(d-1)^2}{2}\bigr).\]

To find the required basis of the $\Hdg( \Ker \phi)$, recall that $\phi$ is induced by the Poincar\'e-Lefschetz duality, and hence is defined by taking the intersection products:
\[\phi \colon \alpha \mapsto (\alpha \cdot \underline{\;\;})_{\tilde{{\mathcal{X}}}} \hspace{5mm} (\alpha \in H_2(Y, \mathbb{Q}(-1))).\]
By the transversality, we obtain
\begin{align*}
&(l_i \cdot [\tilde{L}_{i'}])_{\tilde{\mathcal{X}}} = 1 \hspace{5mm}　(i' \neq i)\\
&(e^{ij}_l \cdot [\tilde{L}_{i'}])_{\tilde{\mathcal{X}}} = \begin{cases}
							1 & (i' = i)\\
							0 & (i' \neq i, j).
							\end{cases}
\end{align*}
Moreover, since the graph of  the map $\tilde{X} \to \mathbb{P}^1$ defines an algebraic cycle in $Z^1(\tilde{\mathcal{X}} \times \mathbb{P}^1)$, $Y = \tilde{X}_0$ and $\tilde{X_t}$ are rationally equivalent. Since $\tilde{X_t} \cap V = \emptyset$ for any subvariety $V$ of $Y$,
$(\alpha \cdot Y)_{\tilde{\mathcal{X}}} = (\alpha \cdot \sum \tilde{L_i})_{\tilde{\mathcal{X}}} = 0$ for any $\alpha \in H_2(Y, \mathbb{Q})$.
Thus we also can see that
\begin{align*}
&(l_i \cdot [\tilde{L}_{i}])_{\tilde{\mathcal{X}}} = (l_i \cdot [Y] - [\sum_{i' \neq i}\tilde{L}_{i'}])_{\tilde{\mathcal{X}}} = -(d-1)\\
&(e^{ij}_l \cdot [\tilde{L}_{j}])_{\tilde{\mathcal{X}}} = (e^{ij}_l \cdot [Y] - [\sum_{j' \neq j}\tilde{L}_{j'}])_{\tilde{\mathcal{X}}} = -1.
\end{align*}
Summarizing the computation, we obtain the following intersection matrix:

$$
\arraycolsep3pt
\left(
\begin{array}{@{\,}c|cccccccccccc@{\,}}
&l_1&l_2& \ldots &l_d & e^{12}_1 & e^{12}_2 & \ldots & e^{12}_{d-1} & e^{13}_1 & e^{13}_2 & \ldots & e^{d-1,d}_{d-1} \\
\hline
\tilde{L}_1&-(d-1)&1&\ldots &1&1&1& \ldots&1&1& 1&\ldots&0\\
\tilde{L}_2&1&-(d-1)&\ldots&1&-1&-1&\ldots&-1&0& 0&\ldots&0\\
\tilde{L}_3&1&1&\ldots&1&0&0& \ldots&0&-1& -1 &\ldots&0\\
\vdots &  & & & & & \vdots & & & & & &\\
\tilde{L}_{d-1}&1&1&\ldots&1&0&0&\ldots&0&0&0 &\ldots&1\\
\tilde{L}_{d}&1&1&\ldots&-(d-1)&0&0&\ldots&0&0&0&\ldots&-1\\
\end{array}
\right)
$$

Note that we do not need to consider $e^{ij}_d$, since it is generated by the above other classes via the relation $l_j - l_ i = \sum_l e^{ij}_l$.
It is easy to check that $\mathcal{B}$ in the statement is a basis of the kernel of this matrix.
In fact, by the intersection products we can see that the $1 + (d-1)\binom{d}{2}$ elements of $\mathcal{B}$ are linearly independent.
On the other hand, since the columns for $(e^{1j}_1 )_{j}$ generate any other columns, ${\rm dim}(\Hdg (\Ker (\phi))) = d \bigl(1 + \frac{(d-1)^2}{2}\bigr) - (d-1) = 1 + (d-1)\binom{d}{2}$. 
\end{proof}

We use the following lemma later to prove Lemma $\ref{computation lemma}$.

\begin{lem}
\label{blow up computation}
	the equation $L_{\sigma(\alpha)} = 0$ in $\mathcal{N}_{il} \coloneqq \widetilde{\mathcal{L}_i \cap \mathcal{M}_l} \subset \tilde{\mathcal{X}}$ defines the divisor class
	\begin{align*}
		& e^{\alpha, \sigma(\alpha)}_l \times \{ 0 \} + [p^{\alpha, \sigma(\alpha)}_l \times \mathbb{P}^1] \quad {\rm if } \quad \alpha < \sigma(\alpha)\\
		&[p^{\sigma(\alpha), \alpha}_l \times \mathbb{P}^1] \quad {\rm if } \quad \alpha > \sigma(\alpha).
	\end{align*}
\end{lem}

\begin{proof}
		Since the statement is local, we take the local coordinates $(x, y, z; t) \in \mathbb{A}^3 \times \mathbb{P}^1$ on an analytic open set $\mathcal{U} \subset \mathcal{X}$ about $p^{\alpha, \sigma(\alpha)}$.
	By taking a sufficiently small $\mathcal{U}$, we may assume that
	\[x = L_{\alpha}, \quad y = L_{\sigma(\alpha )}, \quad z = M_l \]
	and
	\[ \mathcal{U}= \{ xy + tz = 0 \} \subset \mathbb{A}^3 \times \mathbb{P}^1.\]
	 In particular, $p^{\alpha, \sigma(\alpha)}$ is the unique singular point in $\mathcal{U}$.
	
	If we assume $\alpha < \sigma(\alpha)$, then only the $\alpha$-th blow up $b_{\alpha}$ changes $\mathcal{U}$. Note that $b_{\sigma(\alpha)}$ is isomorphic over $\mathcal{U}$ since the node $p^{\alpha, \sigma(\alpha)}$ has already been resolved.
	Therefore, the strict transformation $\tilde{\mathcal{U}} \subset \tilde{\mathcal{X}}$ of $\mathcal{U}$ is isomorphic to the strict transformation $Xy + tZ = 0$ via the blow up of $\mathbb{A}^3 \times \mathbb{P}^1$ along $x = z = 0$.
	Here, $[ X : Z ] \in \mathbb{P}^1$ is the blow up coordinate.
	Then $\mathcal{N}_{il} =  \{ x = z = Xy + tZ= 0\} \subset \mathbb{A}^3 \times \mathbb{P}^1 \times \mathbb{P}^1$
	defines a smooth curve for each $t \neq 0$, but it degenerates to two lines $\{x = y = z = 0\}$ and $\{x = z = X = 0\}$ as $t$ goes to $0$.
	Hence the function $L_{\sigma (\alpha)}=y = 0$ defines two lines $\{x = y = z = t = 0\}$ and $\{x = y = z = X = 0\}$.
	The former one is $E^{\alpha, \sigma(\alpha)}_l \times \{ 0 \}$ and the latter is $p^{\alpha, \sigma(\alpha)}_l \times \mathbb{P}^1$.
	
	If $\alpha > \sigma(\alpha)$, we should change $x$ and $y$ in the above discussion. Specifically, we may assume that $\mathcal{N}_{il}$ is defined by the closure of
	$\{x = z = xY + tZ = Yz - Zy = 0\} \cap \bigr( \mathbb{A}^3 \times \mathbb{P}^1 \times \mathbb{P}^1 \setminus \{y = z = 0\} \bigl) = \{ x = z = Z = 0\} \cap \{ y \neq 0\}$.
	Hence $\mathcal{N}_{il} = \{ x = z = Z = 0\}$ and the function $L_{\sigma (\alpha)}=y = 0$ defines only one line $\{ x = y = z = Z = 0\} = p^{\sigma(\alpha), \alpha}_l \times \mathbb{P}^1$.
\end{proof}

By taking the strict transformation of each subvariety in $\mathcal{X}^* \times \mathbb{P}^n$ by the blow up $\tilde{\mathcal{X}} \times \mathbb{P}^n \to \mathcal{X} \times \mathbb{P}^n$, we obtain higher cycles $\tilde{\mathfrak{Z}} \in CH^p(\tilde{\mathcal{X}}^*, n)$ from each $ \mathfrak{Z} \in CH^p(\mathcal{X}^*, n)$.
We denote
$\widetilde{sing}_{0}(\mathfrak{Z}) \coloneqq \widetilde{sing}_{0}(\tilde{\mathfrak{Z}})$
\begin{remark}
Since $\tilde{\mathcal{X}}^*$ is isomorphic to $\mathcal{X}^*$, $\mathcal{H}^{p,n}_{\pi} \cong \mathcal{H}^{p,n}_{\tilde{\pi}}$.
Moreover, $\nu_{\tilde{\mathfrak{Z}}} = \nu_{\mathfrak{Z}}$ via this isomorphism by the functoriality of the Abel-Jacobi map.
Hence we can identify their singular invariants.
\end{remark}

\begin{lem}
\label{computation lemma}
For each $1 \le i < j < k \le d$ and $1 \le l \le d$,
\begin{enumerate}
	\item 
	$\widetilde{sing}_{0}(\gamma_{ijk, l}) = e^{ij}_l + e^{jk}_l -e^{ik}_l$\\
	\item
	$\widetilde{sing}_{0}(\lambda_{il}) = l_i - \bigl( \sum_{i' < i}e^{i'i}_l  \bigr) + (\sum_{i' > i}e^{ii'}_l).$
\end{enumerate}
\end{lem}

\begin{proof}
	(i) Recall that the higher cycle $\gamma_{ijk,l}$ is constructed by the graph in $\mathcal{X}^* \times \mathbb{P}^1$ of $\phi_{\alpha l}$.
	This rational function also defines a graph $\Gamma_{\phi_{\alpha l}}$ on $\mathcal{X} \times \mathbb{P}^1$.
	Let $\tilde{\Gamma}_{\phi_{\alpha l}}$ be the strict transformation of $\Gamma_{\phi_{\alpha l}}$, so that $\tilde{\gamma}_{ijk,l} = (\tilde{\Gamma}_{\phi_{il}} + \tilde{\Gamma}_{\phi_{jl}} + \tilde{\Gamma}_{\phi_{kl}})|_{\mathcal{X}^*}$.
	By the definition of $res \colon CH^2(\tilde{\mathcal{X}}^*, 1) \to CH^1(\tilde{\mathcal{X}}_{sing})$, $res(\tilde{\gamma}_{ijk,l}) = \partial(\tilde{\Gamma}_{\phi_{il}} + \tilde{\Gamma}_{\phi_{jl}} + \tilde{\Gamma}_{\phi_{kl}}) $.
	
	We now compute $\partial(\tilde{\Gamma}_{\phi_{\alpha l}})$.
	For $\bullet = 0, \infty$, let $\rho_{\bullet}\colon \tilde{\mathcal{X}} \hookrightarrow \tilde{\mathcal{X}} \times \mathbb{P}^1$ be the natural embedding at $0$ or $\infty$.
	Then, by the definition of $\phi_{\alpha l}$, the pullback $\rho_{\bullet}^{*}(\tilde{\Gamma}_{\phi_{\alpha l}}) =(\pi_{\tilde{\mathcal{X}}})_{*}(\tilde{\Gamma}_{\phi_{\alpha l}} \cdot \tilde{\mathcal{X}} \times \{\bullet \})_{\tilde{\mathcal{X}} \times \mathbb{P}^1}$ is given by the equation
	\begin{align*}
		L_{\sigma(\alpha)} = 0 \quad {\rm for } \quad \bullet = 0\\
		L_{\sigma^2(\alpha)} = 0 \quad { \rm for } \quad \bullet = \infty
	\end{align*}
	over $\mathcal{N}_{il} \subset \tilde{\mathcal{X}}$.
	Therefore, by applying Lemma \ref{blow up computation} with $\alpha = i, j, k$ respectively, we obtain
	\[\partial(\tilde{\Gamma}_{\phi_{il}}) = [e^{ij}_l \times \{ 0 \} ]+  [p^{ij}_l \times \mathbb{P}^1] - [e^{ik}_l \times \{ 0 \} ] -  [p^{ik}_l \times \mathbb{P}^1]\]
	\[\partial(\tilde{\Gamma}_{\phi_{jl}}) = [e^{jk}_l \times \{ 0 \} ]+  [p^{jk}_l \times \mathbb{P}^1] -  [p^{ij}_l \times \mathbb{P}^1]\]
	\[\partial(\tilde{\Gamma}_{\phi_{kl}}) = [e^{ik}_l \times \{ 0 \} ] - [e^{jk}_l \times \{ 0 \} ] -  [p^{jk}_l \times \mathbb{P}^1]\]
	and hence
	\[\partial(\tilde{\Gamma}_{\phi_{il}} + \tilde{\Gamma}_{\phi_{jl}} + \tilde{\Gamma}_{\phi_{kl}}) = e^{ij}_l + e^{jk}_l -e^{ik}_l\]
	via the identification of $X_0 \times \{0\}$ with $X_0$.
	
	(ii) Since $\lambda_{il}$ is defined with the rational function $\pi \colon \mathcal{X} \to \mathbb{P}^1$, $\partial(\tilde{\lambda}_{il})|_{X_0} = \widetilde{L_i \cap M_j} \subset \tilde{X}_0$.
	In the proof of Lemma \ref{blow up computation}, we have seen that $\mathcal{N}_{il}$ degenerates to 
	$(d-i)$ exceptional curves $\{ E_{i'l} \}_{i' > i}$ and the other one $C_{il} \subset \tilde{L}_i$, which is isomorphic to the strict transformation of $L_i \cap M_l$ by the blow up of $L_i$ at $d(i-1)$ points $\{p^{i'i}_l\}_{i' < i}$.
	By the intersection products
	\[(l_i \cdot [C_{il}])_{\tilde{L}_i} = 1\]
	\[(e^{i'i}_l \cdot [C_{il}])_{\tilde{L}_i} = 1,\]
	we can see that $[C_ {il}] = l_i - \bigl( \sum_{i' < i}e^{i'i}_l  \bigr)$ in $\Pic (\tilde{L}_i)$.
	Therefore
	\[[\widetilde{L_i \cap M_j}] = [C_{il}] + [\bigoplus_{i' > i}  E_{i'l} ] = l_i - \bigl( \sum_{i' < i}e^{i'i}_l  \bigr) + (\sum_{i' > i}e^{ii'}_l)\]
	in $\Pic (\tilde{X}_0)$.
	\end{proof}

\begin{proof}[Proof of Theorem \ref{main thm}]
From Corollary \ref{Cor for sing} and the remark above Lemma \ref{computation lemma}, we should show that the basis in Lemma $\ref{basis}$ is in $\widetilde{sing}_{0}(\mathscr{I_{\ast} \cup \mathscr{D}})$.
As mentioned before Theorem \ref{main thm}, the discussions for $\widetilde{sing}_{0}$ and $\widetilde{sing}_{\infty}$ are exactly the same.
Hence we prove only for the case $\ast = 0$.

Firstly, we can see that
\[\widetilde{sing}_{0}(\sum_{1 \le i \le d} \lambda_{il}) =  \sum_{1 \le i \le d} l_i\]
for each $l$ by Lemma \ref{computation lemma} and direct computation.
Hence it suffices to show that $\sum_{1 \le l' \le d}(e^{ij}_l - e^{ij}_{l'})$ is in the span of $\widetilde{sing}_{0}(\mathscr{I_{\ast} \cup \mathscr{D}})$ for each $1 \le i<j \le d$ and $1 \le l,l' \le d$.
Since $l_j -l_i = \sum_{1 \le l' \le d} e^{ij}_{l'}$ in $H_2(Y, \mathbb{Q})^{(-1,-1)}$, $\widetilde{sing}_{0}(\lambda_{il} - \lambda_{jl})$ is
\begin{align*}
	 (l_i -l_j) + & 2e^{ij}_l -\bigr(\sum_{k<i}e^{ki}_l - e^{kj}_l\bigl) + \bigr(\sum_{i<k<j}e^{ik}_l + e^{jk}_l\bigl) + \bigr(\sum_{j<k}e^{ik}_l - e^{jk}_l\bigl)\\
	&= -\bigr( \sum_{1 \le l' \le d} e^{ij}_{l'} \bigl) + de^{ij}_l -\bigr(\sum_{k<i}e^{ki}_l -e^{ij}_l - e^{kj}_l\bigl)\\
	 & \qquad + \bigr(\sum_{i<k<j}e^{ik}_l + e^{jk}_l - e^{ij}_l\bigl) + \bigr(\sum_{j<k} - e^{ij}_l +e^{ik}_l - e^{jk}_l\bigl)\\
	&= - \bigr( \sum_{1 \le l' \le d}e^{ij}_l - e^{ij}_{l'} \bigl) - \sum_{k < i} \gamma_{kij,l} + \sum_{i < k < j} \gamma_{ikj,l} - \sum_{j < k} \gamma_{ijk,l} .
\end{align*}
Therefore
\[\sum_{1 \le l' \le d}(e^{ij}_l - e^{ij}_{l'}) = \widetilde{sing}_{0}(\lambda_{jl} - \lambda_{il} - \sum_{k < i} \gamma_{kij,l} + \sum_{i < k < j} \gamma_{ikj,l} - \sum_{j < k} \gamma_{ijk,l} ).\]
\end{proof}

From the above computations, specifically we obtain the expression of the singularities relative to the basis $\mathcal{B}$:
\[\widetilde{sing}_{0}(\gamma_{ijk, l}) = \frac{1}{d}\left(\sum_{l'}(e^{ij}_l - e^{ij}_{l'}) + \sum_{l'}(e^{jk}_l - e^{jk}_{l'}) - \sum_{l'}(e^{ik}_l - e^{ik}_{l'})\right)\]
\[\widetilde{sing}_{0}(\lambda_{il}) = \frac{1}{d} \left( \sum_{i} l_i - \sum_{i' < i} \left( \sum_{l'}(e^{ii'}_l - e^{ii'}_{l'}) \right) + \sum_{i' > i} \left( \sum_{l'}(e^{ii'}_l - e^{ii'}_{l'}) \right) \right) \]
Hence the singularity of $\gamma_{ijk, l}$ is linearly independent from the singularities of $\mathscr{D}$. This implies

\begin{cor} \label{gamma is indecomposable cycle.}
	Each of $\gamma_{ijk, l}$ is an $\mathbb{R}$-regulator indecomposable cycle.
\end{cor}

\section{Higher Chow Cycles with Non-trivial Limits}
\label{Higher Chow Cycles with Non-trivial Limits}

Theorem \ref{basis} shows that $\mathscr{I}_0$ and $\mathscr{I}_{\infty}$ have non-trivial singularities, but at different singular fibers $X_{0}$ and $X_{\infty}$ respectively. From the construction of each higher cycle in $\delta_{i,lmn}$ in $\mathscr{I}_{\infty}$, it is clear that its singularity at $X_{0}$ is trivial. To show the linearly independence of $\mathscr{I}_{0} \cup \{\delta_{i,lmn}\}$,  we compute the limit invariant of $\delta_{i,lmn}$ at $X_{0}$.

We use the same notation $Y = \tilde{X}_0$, $ Y_I = \bigcap_{i \in I} \tilde{L}_i$ as before and also denote $Y^I \coloneqq \bigcup_{j \notin I } Y_{I \cup \{j\}}$
Recall that the motivic cohomology of the simple normal crossing divisor $Y = \tilde{X}_0$ is obtained by
\[H^{2p-n}_{\mathcal{M}}(Y, \mathbb{Q}(p)) = H^{-n}(Z^{\bullet}_Y(p)).\]
Here, we take a subgroup $Z^p_{\#}(Y_I, \bullet) \coloneqq Z^p_{\mathbb{R}}(Y_I, \bullet)_{Y^I} \subset Z^p(Y_I, \bullet)$ which consists of the precycles in good position with respect to $Y^I$ (\cite{KL07}, Section 8) and $Z^{\bullet}_Y(p)$ is the associated simple complex to the double complex
\[Z^{k,m}_{Y}(p) = \bigoplus_{|I| = k + 1} Z^p_{\#}(Y_I, -m)\]
with Bloch's differential $\partial_{\mathcal{B}}$ and the alternating sum $\partial_{\mathcal{I}}$ of the pullbacks by the inclusion $Y_{I \cup \{j\}} \hookrightarrow Y_I$.
Similarly,  the normal currents $D_{\#}^{\bullet}(Y_I)$ and integral currents $C^{\bullet}_{\#}(Y_I, \mathbb{Q}(p))$ denotes the associated simple complex $K^{\bullet}_Y(p)$ of a double complex
\[K_{Y}^{k,m} \coloneqq \bigoplus_{|I|= k + 1}　\{ C^{2p+m}_{\#}(Y_I, \mathbb{Q}(p)) \oplus F^pD^{2p+m}_{\#}(Y_I) \oplus D_{\#}^{2p+m-1} \}\]
and the Deligne cohomology can be obtained by 
\[ H^{2p-n}_{\mathcal{D}}(Y, \mathbb{Q}(p)) = H^{-n}(K^{\bullet}_Y(p)).\]
When the class of a higher cycle $Z$ in $H^{2p-n}_{\mathcal{M}}(Y, \mathbb{Q}(p))$ can be represented by 
\[\{ Z^{[k]}_{I} \in Z^{k,-k-n}_Y(p)\}_{k, |I| = k + 1},\]
 the componentwise KLM formula
 \[\{ (2 \pi i)^{p-k}((2 \pi i)^{k}T_{Z^{[k]}_I}, \Omega_{Z^{[k]}_I}, R_{Z^{[k]}_I})\}\]
in $K^{-n}_Y(p)$ induces $AJ^{p,n}_Y(Z)$.

Now, the strict transformation $\widetilde{\delta}_{i,lmn}$ in $Y_I$ of $\delta_{i,lmn}$ satisfies this condition and hence we can show
\begin{lem}
\label{limit invariant lemma}
When $d = 4$,
$AJ^{2,1}_{Y}(\widetilde{\delta}_{i,lmn})$ is non-trivial in $H^{3}_{\mathcal{D}}(Y, \mathbb{Q}(1))$.

\end{lem}

\begin{proof}
Consider the moduli of the families $\mathcal{X}$.
Since the choices of the linear forms are general, it suffices to show the non-triviality of $AJ^{2,1}_{Y}(\widetilde{\delta}_{i,lmn})$ for a particular family in this moduli space.
For the simplicity we assume that $i = 4, l=1, m=2, n=3$, and we choose the linear forms
\[
\begin{array}{lll}
L_1 \colon X=0 && M_1\colon X+\mu Y - Z + W = 0\\
L_2 \colon Y=0 && M_2\colon \mu X - Y + Z + W = 0\\
L_3 \colon Z=0 && M_3\colon -X + Y + \mu Z + W = 0\\
L_4 \colon W=0 && M_4\colon X + Y + Z - \mu W = 0.
\end{array}
\]
Here, $\mu$ is the primitive 6th root of unity $\frac{1 + \sqrt{3}i}{2}$.
This family is an example of a \textit{tempered} family, a notion which is defined in Section 3 of \cite{DK11}  for more general toric hypersurfaces. For our case, this condition is  equivalent to each $p^{lm}_{i}$ having the root of unity coordinates with respect to $[X:Y:Z:W].$ A crucial point of the smooth tempered families of toric hypersurafaces which is defined by a reflexive Newton polytope is that the natural Hodge class 
\[\frac{1}{(2 \pi i)^n}d \log x_1 \wedge d\log x_2 \wedge \ldots d \log x_n  \in H^n((\mathbb{C^*})^n, \mathbb{Q}(n))\]
 defined by the toric coordinate symbol $\{x_1, x_2, \ldots, x_n \} \in  H^n_{\mathcal{M}}((\mathbb{C^*})^n, \mathbb{Q}(n))$ can be extended to the Hodge class on the family itself. Therefore if we take 2-form 
 \[\omega \coloneqq \left(\frac{1}{\twist}\right)^2 \frac{dx}{x} \wedge \frac{dy}{y}\]
 with $(x, y) \coloneqq (X/Z, Y/Z)$ for each general fiber $X_t$, dually it defines a family $ \{\omega(t) \in \Hdg(H_{2}(X_t, \mathbb{Q}(-2)))\}$.
By the KLM formula, we shall compute the membrane integral of this test 2-form $\omega$ on the triangle $\Gamma$ whose edges are the strict transformations of the three lines
\[
\begin{split}
	&L_4 \cap M_1  \colon X+\mu Y - Z = 0\\
	&L_4 \cap M_2  \colon \mu X+ Y + Z = 0\\
	&L_4 \cap M_3  \colon -X + Y + \mu Z = 0\\
\end{split}
\]
coming from $\widetilde{\delta}_{i,lmn}$. Then
\begin{align*}
	AJ_{i,lmn} \coloneqq AJ^{2,1}_{Y}(\widetilde{\delta}_{i,lmn})(\omega) &= (- 2\pi i) \left( \int_{L_i}  R_{\widetilde{\delta}_{i,lmn}} \wedge \omega + (2\pi i )\int_{\Gamma}\omega \right)\\
	&= \left( (- 2\pi i)\int_{\widetilde{\delta}_{i,lmn}}  \log(t) \omega \right) -  \left( (2\pi i )^2\int_{\Gamma}\omega \right) .
\end{align*}
However, its first term vanishes since $dx$ and $dy$ are linearly dependent on $\widetilde{\delta}_{i,lmn}$.
Since the vertices of $\Gamma$ with respect to the coordinates $(x,y)$ are given by
\[
\begin{matrix}
p^{lm}_{i}=(-\mu,2 - \mu ), & p^{ln}_{i}=(i \sqrt3,\mu^2), & p^{mn}_{i}=(\frac13 (1 + \mu),-\frac{1}{\sqrt3}i),
\end{matrix}
\]
we obtain \begin{align*}
	- AJ_{i,lmn} &= \int_{\Gamma} \frac{dx}{x} \wedge \frac{dy}{y}\\
	&= \int ^{2-\mu}_{-\frac{1}{\sqrt3}i} \left( \int^{-\mu y + 1}_{y + \mu}\frac{dx}{x} \right)\frac{dy}{y} +  \int ^{\mu^2}_{2-\mu} \left( \int^{\frac{1}{\mu}y - \frac{1}{\mu}}_{y + \mu}\frac{dx}{x} \right)\frac{dy}{y} \\
&= \left(\int ^{2-\mu}_{-\frac{1}{\sqrt3}i} \frac{\log(-\mu y + 1)}{y} dy \right)
+ \left(\int ^{\mu^2}_{2-\mu} \frac{\log (\frac{1}{\mu}y - \frac{1}{\mu})}{y}dy \right)\\ 
& \qquad - \left(\int ^{\mu^2}_{-\frac{1}{\sqrt3}i} \frac{\log(y + \mu)}{y} dy\right).
\end{align*}
Generally, the integral of a multivalued function $\frac{\log (a + bz)}{z}$ ($a, b \in \mathbb{C}$) is
\[
\int \frac{\log (a + bz)}{z}dz = -\Li_2(-\frac{b}{a}z) + \log(z)\left( \log(a + bz) - \log(1 + \frac{b}{a}z) \right).
\]
with the dilogarithm function $\Li_2$. By applying this integral to each term of $- AJ_{i,lmn}$, we can see that
\begin{align*}
	&\int ^{2-\mu}_{-\frac{1}{\sqrt3}i} \frac{\log(-\mu y + 1)}{y} dy = -\Li_2 \left(1 + \mu \right) + \Li_2 \left(\frac{1}{1 + \mu}\right) \\
	& \int ^{\mu^2}_{2-\mu} \frac{\log (\frac{1}{\mu}y - \frac{1}{\mu})}{y}dy = \left( -\Li_2 \left(\frac{1}{- \mu} \right) + \Li_2 (1 - \mu^2)  \right) 
	+ \left( - \frac{2}{9}\pi^2 + \frac23 i \pi \log3 \right) \\
	& -\int ^{\mu^2}_{-\frac{1}{\sqrt3}i} \frac{\log(y + \mu)}{y} dy= \left( \Li_2(-\mu) - \Li_2 \left(\frac{1}{1 - \mu^2}\right)  \right) 
	+ \left( \frac{7}{18}\pi^2 - \frac16 i \pi \log3 \right).
	\end{align*}
and hence
\begin{align*}
- AJ_{i,lmn} = &\Li_2(-\mu) - \Li_2\left(\frac{1}{- \mu}\right) + \Li_2 \left(\frac{1}{1 + \mu}\right) - \Li_2 \left(1 + \mu \right)  \\
& + \Li_2 (1 - \mu^2)- \Li_2 \left(\frac{1}{1 - \mu^2} \right) + \left( \frac{1}{6}\pi^2 + \frac12 i \pi \log(3) \right).
\end{align*}	
	
To compute the dilogarithm terms, we also use functional equations
\[ \Li_2 \left( \frac{z-1}{z} \right) - \Li_2 \left( z \right) = -\frac16 \pi^2 + \log (z) \log(1-z) - \frac12 \log(z)^2 \]
\[\Li_2 \left( \frac{1}{1-z} \right) - \Li_2 \left( z \right) = \frac16 \pi^2 + \log(-z) \log(1-z) - \frac12 \log(1-z)^2,\]
for $z$ which is not on the branch cuts. Note that $\mu$ satisfies the equations
\[
\frac{1}{1 + \mu} =\frac{1}{1- (-\mu)},　\; 1 + \mu = \frac{(-\frac{1}{\mu}) - 1}{-\frac{1}{\mu}}, \;
1 - \mu^2 = \frac{(- \mu) -1}{-\mu}, \; \frac{1}{1 - \mu^2} = \frac{1}{1 - (-\frac{1}{\mu})}.
\]
Hence we can show
\begin{align*}
\Li_2 \left(\frac{1}{1 + \mu}\right) - \Li_2 \left(1 + \mu \right) =  & \Li_2(-\mu) - \Li_2\left(\frac{1}{- \mu}\right)\\
& + \frac13 \pi^2 - \frac38\pi^2 - \frac{\log(3)^2}{8} - \frac14 i \pi \log(3) \\
\Li_2 (1 - \mu^2) - \Li_2 \left(\frac{1}{1 - \mu^2}\right) = & \Li_2(-\mu) - \Li_2\left(\frac{1}{- \mu}\right) - \frac13 \pi^2 \\
&+ \frac38\pi^2 + \frac{\log(3)^2}{8} - \frac14 i \pi \log(3).
\end{align*}
With $-\frac{1}{\mu} = \overline{-\mu}$, finally we obtain
\begin{align*}
- AJ_{i,lmn} &= 3(\Li_2(-\mu) - \Li_2(\overline{-\mu})) + \left( - \frac12 i \pi \log (3) \right) + \left( \frac{1}{6}\pi^2 + \frac12 i \pi \log(3) \right)\\
&= 3(\Li_2(-\mu) - \overline{\Li_2(-\mu)}) + \zeta(2).
\end{align*}
Since the first term is  purely imaginary and non zero, it shows that $AJ_{i,lmn}$ is non-trivial in $\mathbb{C} / \mathbb{Q}(2)$.
\end{proof}

Since $ \{\omega(t) \}$ is the family of Hodge classes, as we see at the end of Section \ref{Singularities and limits of Normal Functions}, we have
\[\lim_{t \to 0}\langle\nu_{\mathfrak{\delta_{i.lmn}}}(t), \omega(t)\rangle \equiv AJ_{i,lmn} \in \mathbb{C} / \mathbb{Q}(2).\]
From the above lemma, the right hand side is non-trivial and hence we have proven 

\begin{thm} \label{limit invariant}
Suppose $d = 4$. For general choices of $\{L_i\}$ and $\{M_l\}$, $\nu_{\delta_{i.lmn}}$has non-trivial limit. Especially the higher cycles $\{\delta_{i.lmn}\} \cup \mathscr{I}_0 \cup \mathscr{D}$ are linearly independent in $CH^2(X_t,1)$.
\end{thm}

\section{Hodge-$\mathcal{D}$-Conjecture for a certain type of $K3$ surfaces}
\label{Application: Hodge-D-Conjecture for a certain type of K3 surfaces}

In this section we consider the case that $d=4$. Hence $X_t$ is a $K3$ surface with the form
\[X_t \colon L_1L_2L_3L_4 + t M_1M_2M_3M_4 = 0,\] 
and $H^{3}_{\mathcal{D}}(X_t, \mathbb{R}(2)) \cong H^{1,1}_{\mathbb{R}}(X_t)(1)$ is 20-dimensional. Though the real regulator map is generally not injective, by computing the limit of real regulator values we can see that the image of 20 higher cycles $\{\delta_{i.lmn}\} \cup \mathscr{I}_0 \cup \mathscr{D}$ actually spans this vector space.

\begin{thm} \label{Hodge-D-Conjecture}
	When $d=4$ and $\{L_i\}, \{M_l\}$ are very general, $ r^{2,1}_{\mathcal{D}, \mathbb{R}}(\{\delta_{i.lmn}\} \cup \mathscr{I}_0 \cup \mathscr{D} ) $ are linearly independent in $H^{1,1}_{\mathbb{R}}(X_t)(1)$, explicitly validating the Hodge-$\mathcal{D}$-Conjecture this case.
\end{thm}

\begin{proof}
	Since $\Hdg ( \Coker (N))$ is 19 dimensional, $\Hdg(\Ker (N))$ is also 19 dimensional.
For a fixed $X_t$, take a basis $d_1, \ldots, d_{19}$  of $\Hdg(\Ker (N))$.
By Theorem \ref{main thm}, the images of linear combinations of higher cycles in $\mathscr{I}_0 \cup \mathscr{D}$ give these classes.
We also take an element $\gamma_2 \in H_{lim} \coloneqq H^2_{lim}(X_t, \mathbb{Q}(2))$ which does not vanish in $\Gr^{W}_4 H_{lim}$.
Denote $\gamma_1 \coloneqq N \gamma_2$, $\gamma_0 \coloneqq N^2 \gamma_2$.
Though each $\{ \gamma_i \}$ defines a multivalued section of the cohomology sheaf $\mathcal{H}^{2}$, from them we can define single valued sections by
\[e'_i \coloneqq e^{-l(t)N}\gamma_i(t).\]
with $l(t) \coloneqq \frac{log(t)}{2\pi i}$.
Specifically
$	e'_0 = \gamma_0$,
$	e'_1 = \gamma_1 - l(t)\gamma_0$, and 
$	e'_2 = \gamma_2 - l(t)\gamma_1 + \frac{l^2(t)}{2}\gamma_0.
$
Hence $\{e'_0, e'_1, e'_2, d_1, \ldots, d_{19}\}$ is a single valued frame of the extension $\mathcal{H}^{2}_e$.
Since $d_i$ is already a single valued section (in other word, $d_i = e^{-l(t)N}d_i(t)$), we obtain a single valued frame
\[\{e'_0, e'_1, e'_1, d_1, \ldots, d_{19}\}\]
of the cohomology sheaf $\mathcal{H}^2_e$.

Take a holomorphic section $\omega(t) \in F^2(H^2_{lim, \mathbb{C}})$ such that $\omega \neq 0$ in $\Gr^W_4(H^2_{lim, \mathbb{C}})$.
Since $\dim \Gr^W_4 =1$, $\omega$ is a generator of $\Gr^W_4$.
Generally $\omega(t)$ can be written as 
\[\omega(t) = e'_2 + f(t)e'_1 + g(t)e'_0 + \sum_{i =1}^{19}h_i(t)d_i \]
with holomorphic functions $f(t), g(t), h_i(t)$ by normalizing $\omega$ with respect to the coefficient of $e'_2$.
By changing $t$ to the new coordinate $t' \coloneqq te^{-2\pi if(t)}$
(hence $l(t') = l(t) - f(t)$),
we define $e_i$ from $e'_i$:
\[e_i(t') \coloneqq e^{f(t)N}e_i(t) = e^{-l(t')N}\gamma_i(t).\]
Note that this shift of the parameter does not change $d_i$.
Hence
\begin{align*}
	\omega(t) &= (\gamma_2 - l(t)\gamma_1 + \frac{l^2(t)}{2}\gamma_0) + f(t)(\gamma_1 - l(t)\gamma_0) + g(t)\gamma_0 + \sum h_i(t)d_i \\
	&= (\gamma_2 -l(t')\gamma_1 + + \frac{l^2(t')}{2}\gamma_0) + (g(t)-\frac{f^2(t)}{2})\gamma_0 + g(t)\gamma_0 + \sum h_i(t)d_i \\
	&= e_2 + (g(t)-\frac{f(t)^2}{2})e_0 + \sum h_i(t)d_i.
\end{align*}
For the simplicity, we use the notation $t$ for $t'$ instead of the original coordinate from here.
Then, by changing $f, g, h_i$ to new functions, we can write $\omega$ as
\[\omega(t) = e_2 + g(t)e_0 + \sum h_i(t)d_i = e_2 + \kappa(t) \]
with $\kappa(t) \coloneqq g(t)e_0 + \sum h_i(t)d_i \in \Ker N$.
Though $e_2$ may not  be in $F^2(H^2_{lim, \mathbb{C}})$, we obtain
\[e_1 = Ne_2 = N\omega\]
since $\kappa(t) \in \Ker N$.
Hence $e_1 \in F^1 \cap W_2(H^2_{lim, \mathbb{C}})$ and $e_0 = Ne_1 \in F^0 \cap W_0(H^2_{lim, \mathbb{C}})$.

By the definition, it is easy to check that the quadratic form $Q(e_i, e_j)$ for the polarization is given by the matrix
\[
\begin{pmatrix}
	0& 0 & -1\\
	0 & 1 & 0\\
	-1 & 0 &0
\end{pmatrix}
\]
after a normalization of $\gamma_2$.
Also note that the conjugates satisfy the equalities
\begin{align*}
\overline{e_0} &= e_0\\
\overline{e_1} &= e_1 + 2i\Im(l)e_0\\
\overline{e_2} &= e_2 + 2i\Im(l)e_1 - 2(\Im(l))^2e_0\\
\overline{d_i} &= d_i
.
\end{align*}
with the imaginary part $\Im(l) = -\frac{\log |t|}{2 \pi}$ of $l(t)$.

Take a non zero element $\eta \in H^{1,1}_{lim, \mathbb{R}}$ which is linearly independent from $d_1, \ldots, d_{19} \in H^{1,1}_{lim, \mathbb{R}}$.
We shall express $\eta$ by using $e_0, e_1, e_2$.
Since $\eta \in F^1$, there exists a $C^{\infty}$ function $\phi (t)$ such that
\[\eta = \omega + \phi (t) e_1 = e_2 + \phi(t) e_1 + \kappa(t).\]
$\eta$ is also a real form, hence
\[
\eta = \overline{\eta} = (e_2 + 2i\Im(l)e_1 - 2(\Im(l))^2e_0) + \overline{\phi(t)}(e_1 + 2i\Im(l)e_0) + \overline{\kappa(t)} \in F^1(H^2_{lim}).
\]
Specifically $\overline{\kappa(t)} = \overline{g(t)}e_0 + \sum \overline{h_i(t)}d_i$, hence this term does not include any $e_1$ term.
Thus we can compare the $e_1$ terms of $\eta$ and $\overline{\eta}$ to obtain $\Im(\phi(t)) = \Im(l)$.
Also, the coefficient of the $e_0$ term of $2\eta = \eta + \overline{\eta}$ is given by $2\Re(g(t)) - 2(\Im(l))^2 + 2i\overline{\phi(t)}\Im(l)$, which must be a real number.
This implies that $\phi(t)$ is pure imaginary.
Therefore
\[\phi(t) = i\Im(l)\]
and hence
\[\eta = e_2 + i\Im(l)e_1 + \kappa(t).\]

Now, we compute the rational regulator $R(t)$ of $\delta_{i,lmn}$.
Then the real regulator is given by $\Im(R(t))$ via the isomorphism $H^3_{\mathcal{D}}(X_t, \mathbb{R}(2)) \cong H^{1,1}(X_t, \mathbb{R})\otimes \mathbb{R}(1) $.
We know that the singular invariant of $\delta_{i,lmn}$ is trivial and its limit invariant is a pure imaginary number $iL\coloneqq AJ_{i,lmn} \in \mathbb{C}/ \mathbb{Q}(2)$.
Hence we may write
\[R(t) = iLe_0 + t \left(\alpha_0(t)e_0 + \alpha_1(t)e_1 + \alpha_2(t)e_2 + \sum_{j = 1}^{19} \beta_j(t)d_j \right)\]
with holomorphic functions $\alpha_i(t), \beta_j(t)$.
Hence 
\begin{align*}
\Im(R(t)) &= -\frac{i}{2}(R(t) - \overline{R(t)})\\
&= -\frac{i}{2}\left( iLe_0 + t\left(\alpha_0(t)e_0 + \alpha_1(t)e_1 + \alpha_2(t)e_2 + \sum \beta_j(t)d_j \right) \right.\\
&\ \ -(-iLe_0 + \overline{t}\left( \overline{\alpha_0(t)}e_0 + \overline{\alpha_1(t)}(e_1 + 2i\Im(l)e_0) \right.\\
&\ \ \left. \left. + \overline{\alpha_2(t)}(e_2 + 2i\Im(l)e_1 - 2(\Im(l))^2e_0) + \sum \overline{\beta_j(t)}d_j \right) \right)\\
&= \left( L + \Im(t\alpha_0(t)) - \overline{t\alpha_1(t)}\Im(l) -i\overline{t\alpha_2(t)}(\Im(l))^2 \right) e_0 +\\
 &\ \ \left( \Im(t\alpha_1(t)) - \overline{t\alpha_2(t)}\Im(l)\right) e_1 + \Im(t\alpha_2(t))e_2 + \sum \Im (t\beta_j(t))d_j.
\end{align*}

Finally we consider the limit of $Q(\Im(R(t)), \eta)$ as $t \to 0$.
Note that $d_i \in \Hdg(\Ker(N))$ is orthogonal to each of $e_0, e_1, e_2$.
With the notation $q_{ij} \coloneqq Q(d_i, d_j)$, hence we obtain
\begin{align*}
	Q(\Im(R(t)), \eta) &= Q\left(\left( L + \Im(t\alpha_0(t)) - \overline{t\alpha_1(t)}\Im(l) -i\overline{t\alpha_2(t)}(\Im(l))^2 \right) e_0 +\right. \\
 &\ \ \left. \left( \Im(t\alpha_1(t)) - \overline{t\alpha_2(t)}\Im(l)\right) e_1 + \Im(t\alpha_2(t))e_2, e_2 + i\Im(l)e_1 + g(t)e_0 \right) \\
	& \ \ + Q \left(\sum \Im (t\beta_j(t))d_j, \sum h_i(t)d_i \right)\\
	&=-(L + \Im(t\alpha_0(t)) - \overline{t\alpha_1(t)}\Im(l)) + i\Im(t\alpha_1(t))\Im(l) - g(t)\Im(t \alpha_2)\\
	& \ \ + \sum_{i, j}h_i(t)\Im (t\beta_j(t))q_{ij}\\
	&= -L - \Im(t\alpha_0(t)) + \Re(t\alpha_1(t))\Im(l) - g(t)\Im(t \alpha_2) + \sum_{i, j}h_i(t)\Im (t\beta_j(t))q_{ij}.
	\end{align*}
This value goes to $-L$ as $t \to 0$.
On the other hand, for each $d_i$,
\[Q(\Im(R(t)), d_i) = Q(\sum_j \Im (t\beta_j(t))d_j, d_i) = \sum_j \Im (t\beta_j(t))q_{ij}\]
goes to 0 as $t \to 0$.
Hence we conclude that
\begin{align*}
&\lim_{t\to 0}Q(\Im(R(t)), \eta) = -L \\
&\lim_{t\to 0}Q(\Im(R(t)), d_j) = 0.
\end{align*}
This shows that $r^{2,1}_{\mathcal{D}, \mathbb{R}}(\{\delta_{i.lmn}\} \cup \mathscr{I}_0 \cup \mathscr{D})$ are linearly independent.
In fact,  a linear combination of $\mathscr{I}_0 \cup \mathscr{D}$ defines an admissible normal function $R_i(t)$ for each $i$ $(1 \le i \le 19)$ with $sing_0(R_i(t)) = d_i$. This function has a form 
\[
	R_i(t) \coloneqq \alpha_0(t)e_0 + \alpha_1(t)e_1 + \alpha_2(t)e_2 + i \log (t)d_i + \sum_{j \ne i} \beta_j(t)d_j.
\]
and the admissibility implies that each $\alpha_i(t)$ and $\beta_j(t)$ is a holomorphic function 
(\cite{SZ85}, Proposition 5.28).
Therefore 
\[
\lim_{t \to 0} Q(\Im(\frac{1}{\log(t)}R_i(t)), d_j) = 
\begin{cases}
	1 & (i=j)\\
	0 & (i \ne j)
\end{cases}
\]
and moreover
\begin{center}
	 $\displaystyle{\lim_{t \to 0} Q(\Im(\frac{1}{\log(t)}R_i(t)), \eta)}$ is a finite number $C_i$.
\end{center}
In fact, 
\begin{align*}
	\Im \left( \frac{R_i(t)}{\log (t)} \right) &= \left( \Im(\frac{\alpha_0}{\log(t)}) - \Im(l)\overline{\frac{\alpha_1}{\log (t)} } - i (\Im(l))^2\overline{\frac{\alpha_2}{\log (t)}} \right) e_0 \\
	& \ \  + \left( \Im(\frac{\alpha_1}{\log(t)}) -  \Im(l) \overline{\frac{\alpha_2}{\log (t)}} \right) e_1 \\
	& \ \ + \Im(\frac{\alpha_2}{\log(t)})e_2 + d_i  + \sum_{j \ne i}\Im(\frac{\beta_2}{\log(t)})d_j,
\end{align*}
hence the only non-vanishing term of $\Im \left( \frac{1}{\log (t)}R_i(t) \right)$ as $t \to 0$ is 
\[- \left( \Im(l)\overline{\frac{\alpha_1}{\log (t)} }e_0 + \Im(l) \overline{\frac{\alpha_2}{\log (t)}}e_1 + d_i \right).\]
Thus its paring with $d_i$ is 1 and with $d_j$ vanishes for $i \neq j$ by the orthogonality of $d_j$ and each of $e_0, e_1$. Also the pairing with $\eta = e_2 + i\Im(l)e_1 + \kappa(t)$ is finite.
Summarizing them, the matrix defined by paring $\lim_{t\to 0}Q(\underline{\ \ }, \underline{\ \ })$ is given as
$$
\arraycolsep3pt
\left(
\begin{array}{@{\,}c|ccccc@{\,}}
&\eta &d_1& d_2 & \ldots & d_{19} \\
\hline
 \Im (R(t))&-L&0 & 0 & \ldots & 0\\
\Im(\frac{1}{\log(t)}R_1(t))&C_1&1 & 0 & \ldots & 0\\
\Im(\frac{1}{\log(t)}R_2(t))&C_2&0 & 1 & \ldots & 0\\
\vdots & & & & \ddots \\
\Im(\frac{1}{\log(t)}R_{19}(t))&C_{19}&0 & 0 & \ldots & 1\\

\end{array}
\right).
$$
Hence its determinant is non-trivial and it implies that the real regulator value $\Im (R(t)), \Im (R_1(t)), \ldots, \Im (R_{19}(t))$ are linearly independent.
\end{proof}

The above computation of the real regulator values especially show that
\begin{cor} \label{delta is indecomposable cycle.}
	When $d = 4$, each of $\delta_{i,lmn}$ is an $\mathbb{R}$-regulator indecomposable cycle.
\end{cor}

\begin{remark}
(1) If we assume the given VMHS is a nilpotent orbit and the family is tempered, the above computation in the proof is much simpler. In fact these conditions imply $\omega = e_2$ and hence $\eta = e_2  + i \Im(l) e_1$. To compute the pairing with this $\eta$ and each $d_j$ as $t \to 0$, we may assume that
\[R(t) = iL e_0\]
\[\frac{R_i(t)}{\log (t)} = \frac{\alpha_1(t)}{\log(t)} e_1 + i d_i.\]
Hence we obtain exactly the same matrix as above.

\noindent
(2) Though we can construct the family of higher cycles $\delta_{i,lmn}$ even for general $d \ge 5$, there are two problems to apply the similar discussion to prove the Hodge-$\mathcal{D}$-Conjecture.
Firstly, by applying an action of $PGL_3$ which maps the plane $L_i$ and three lines $L_i \cap M_l$, $L_i \cap M_m$, $L_i \cap M_n$ to the special ones in the proof of Lemma \ref{limit invariant lemma},  we may take exactly the same family of test forms $\{\omega(t)\}$. However, this 2-form may not be a Hodge class unless $\mathcal{X}$ is a tempered family after applying the action. Another issue is that generally $\dim (H^{1,1}_{\mathbb{R}}(X_t)) - \dim (\Hdg(\Coker N )) > 1$, hence we need to show not only the non-triviality of $\delta_{i,lmn}$, but the linearly independence of some of $\{r^{2,1}_{\mathcal{D},\mathbb{R}}(\delta_{ijk,l})\}$ (for example, we need four linearly independent classes when $d = 5$). One possible approach to solve this point is to find an enough number of test Hodge classes such that the matrix of the paring $\lim_{t\to 0}Q(\underline{\ \ }, \underline{\ \ })$ is regular.

\end{remark}
 

\section{Application: Threefold with Non-Trivial Griffiths Groups}
\label{Application: Threefold with Non-Trivial Griffiths Groups}

As an application of our main theorem, we construct non-trivial elements of the Griffiths group of a certain threefold which is constructed from $X_t$.

When we consider a proper smooth family $\mathcal{Y}$ over a quasi-projective curve $S$, a given family of cycles $\mathcal{Z}$ in $CH^p(\mathcal{Y})$ such that $Z_s$ is in $CH^p_{hom}(Y_s)$ on each fiber $Y_s$ ($s \in S$) defines the class of $Z_s$ in $\Griff^p(Y_s)$.
The coniveau filtration $N^{\bullet}$ of $H^l(Y_s, \mathbb{Q})$ is defined by
\[N^kH^l(Y_s, \mathbb{Q}) = \sum_{W \subset Y_s, {\rm codim} W \geq k} \Ker(H^l(Y_s, \mathbb{Q}) \to H^l(Y_s \setminus W, \mathbb{Q})).\]
Since $AJ (CH^p_{alg}(Y_s)) \subset N^{p-1}H^{2p-1}(Y_s, \mathbb{Q}(p)),$
the Abel-Jacobi map induces a map
\[ \Griff^p(Y_s) \to J \left( H^{2p-1}(Y_s, \mathbb{Q}(p)) / N^{p-1}H^{2p-1}(Y_s, \mathbb{Q}(p))\right).\]
On the other hand, when the completion $\overline{\mathcal{Y}} \to \overline{S}$ is also proper and $\overline{\mathcal{Y}}$ is smooth, $Z_0$ for each discriminant locus $0 \in \overline{S} \setminus S$ is an element of the motivic cohomology $H^{2p}_{\mathcal{M}}(Y_0, \mathbb{Q}(p))$.
Suppose that $Y_0$ is a SNCD with the strata $Y_0^{[k]}$, then we obtain the induced map
\[CH^p_{\ind}(Y_0^{[1]}, 1) \to J(H^{2p-2}_{tr}(Y_0, \mathbb{Q}(p))) \cong J(H^{2p-2}(Y_0, \mathbb{Q}(p)) / N^{p-1}H^{2p-2}(Y_0, \mathbb{Q}(p))\]
by the Abel-Jacobi map.

The stratification $( Y_0^{[k]} )$ as a semi-simplicial hypercovering structure on $Y_0$ defines a weight filtration $W_{\bullet}$ on $H^{2p-r}_{\mathcal{M}}(Y_0, \mathbb{Q}(p))$.
$W_k$ consists of the classes which can be represented by elements in $\bigoplus_{l \geq -k} Z^{l,-l-r}_{Y_0}(p)$.
When $Z_0 \in W_{-1}H^{2p}_{\mathcal{M}}(Y_0, \mathbb{Q}(p))$, it defines an element of $CH^p_{\ind}(Y_0^{[1]}, 1)$ since the degree 0 term of $Z^{\bullet}_{Y_0}(p)$ is the direct sum of the following boxed components:
\[
\xymatrix{
&\vdots & \vdots &\\
\ldots \ar[r] &  \boxed{ Z^p_{\#}(Y_0^{[1]}, 1)} \ar[r]^{\partial_{\mathcal{B}}} \ar[u] & Z^p_{\#}(Y_0^{[1]}) \ar[r]\ar[u] & \ldots \\
\ldots \ar[r] & Z^p_{\#}(Y_0^{[0]},1) \ar[r]^{\partial_{\mathcal{B}}} \ar[u]^{\partial_{\mathcal{I}}} &  \boxed{ Z^p_{\#}(Y_0^{[0]}) \ar[r] \ar[u]^{\partial_{\mathcal{I}}} } & \ldots \\
&\vdots \ar[u] & \vdots \ar[u] &\\
}
\]
We denote the subgroup of $H^{2p}_{\mathcal{M}}(\mathcal{Y}, \mathbb{Q}(p))$ consisting of the families $\mathcal{Z}$ such that the general $Z_s$ is in $CH^p_{hom}(Y_s)$ and $Z_0$ is in $W_{-1}H^{2p}_{\mathcal{M}}(Y_0, \mathbb{Q}(p))$ by $W_{-1}H^{2p}_{\mathcal{M}}(\mathcal{Y}, \mathbb{Q}(p))'$.
With the analytic limit of the Abel-Jacobi value, hence we obtain the diagram
\[
\xymatrix{
 & \Griff^p(Y_s) \ar[r]^-{AJ} & J \left( \frac{H^{2p-1}(Y_s, \mathbb{Q}(p))}{N^{p-1}H^{2p-1}(Y_s, \mathbb{Q}(p)})\right) \ar[dd]^{\lim_{s \to s_0}}
\\
W_{-1}H^{2p}_{\mathcal{M}}(\mathcal{Y}, \mathbb{Q}(p))' \ar[ru]^{\iota^*_s}\ar[rd]^{\iota^*_0} & &
\\
& CH^p_{\ind}(Y_0^{[1]}, 1) \ar[r]^-{AJ} & J(\frac{H^{2p-2}(Y_0, \mathbb{Q}(p))}{N^{p-1}H^{2p-2}(Y_0, \mathbb{Q}(p)}).
}
\]

Since this diagram is commutative (Section 5.2 of \cite{dDI17}), if $Z_0$ defines an $\mathbb{R}$-regulator indecomposable cycle, it implies that $Z_s$ is non-trivial in $\Griff^k(Y_s)$ for a general $s$.

Now, as the proper smooth family $\mathcal{Y}$, we take a resolution of the singular family $\mathcal{Y}'$ defined as follows:
Firstly take a general $t_0 \in \mathbb{P}^1$ near $0$. Then $X_{t_0}$ is a smooth degree $d$ surface which is discussed in the previous sections.
For simplicity, we denote its defining function by $L + t_0M =0$.
With a new parameter $u \in \mathbb{P}^1$, we obtain a family of degree $d+1$ threefolds
\[\mathcal{Y}' \coloneqq \{Y'_s \colon (L + t_0M)( \frac{u}{u^2-1}) + s(L + t_0M(\frac{u}{u^2-1})) = 0\}\]
in $\mathbb{P}^3 \times \mathbb{P}^1 \times \mathbb{P}^1$.
Clearly $Y'_0$ is a singular fiber consisting of the union of the constant family $X_{t_0}  \times \mathbb{P}^1$ along $u$ and two copies of $\mathbb{P}^3$ at $u = 0, \infty$. 
The singular loci of $\mathcal{Y}'$ must be on the base locus $ L + t_0M = L + t_0M(\frac{u}{u^2 -1}) = 0$, and hence we can compute the Jacobian by taking the local coordinates such as $x = L_i, y = L_j, z = M_s$ for $(x, y, z) \in \mathbb{A}^3$. Thus we can see that the singular locus of 
$\mathcal{Y}'$ near $s=0$ comprises the lines
\[Y'_0 \cap L_i \cap M_l \cap \{u= 0, \infty\}\]
 on $Y_0$ and points
\[Y'_t \cap L_i \cap L_j \cap M_l \cap \{u= 0, \infty\}\]
on every fiber $Y_s$.
We resolve them by the successive blow ups of $\mathbb{P}^3 \times \mathbb{P}^1 \times \mathbb{P}^1$ along the constant family of lines $L_i \cap M_l$ for each combination of $i, l$ and then the family turns to be a semistable degeneration toward $s = 0$.
After resolving the other singularities of $\mathcal{Y}'$, we obtain a smooth family of degree $d + 1$ threefolds $\mathcal{Y}$ over $\overline{S} = \mathbb{P}^1$ which degenerates to a SNCD $Y_0$.
\begin{remark}
In particular this is a degenerating family of Calabi-Yau threefolds when $d = 4$.
	A motivation of this construction comes from the toric geometry. $X_t$ is defined as a Laurent polynomial of a 3 dimensional reflexive Newton polytope $\Delta$. By changing the coordinate $u$ to $w = \frac{u-1}{u+1}$ adjusting the coefficients, we obtain the equation
	\[ (L + t_0M)( w - \frac{1}{w}) + s(L + t_0M(w - \frac{1}{w}) = 0.\]
	This is a Laurent Polynomial with support contained in the Minkowski sum of the interval $[-1,1]$ as a polytope and $\Delta$, and is an example of the construction of Tyurin degenerations of Calabi-Yau threefolds from a nef-partition of a reflexive polytope (\cite{DHT17}, Chapter 3.1).
	It will be a future work to extend this construction to more general nef-partitions.
	
	\end{remark}
Since a component of $Y'_0$ is the constant family $X_{t_0} \times \mathbb{P}^1$, each $\gamma_{ijk, l}$ and $\delta_{i, lmn}$ is also on  $Y'_0$.
We denote their pullback to $Y_0$ by $\widetilde{\gamma_{ijk, l}}$ and $\widetilde{\delta_{i, lmn}}$ respectively.
\begin{thm}\label{non-triviality in Griffiths groups}
	For each higher Chow cycle $\gamma_{ijk, l}$ ($\delta_{i, lmn}$) on $X_{t_0}$, there exists a family of  algebraic 1-cycles $\mathcal{C}_ {ijk,l}$ (resp. $\mathcal{D}_{i,lmn}$) on $\mathcal{Y}$ such that the fiber $(\mathcal{C}_{ijk,l})_0$ in $Y_0$ defines the same class with $\widetilde{\gamma_{ijk, l}}$ (resp. $\widetilde{\delta_{i, lmn}}$) in $H^{4}_{\mathcal{M}}(Y_0, \mathbb{Q}(2))$.
\end{thm}
Since the above successive blow up is isomorphic over the component $X_{t_0} \times \mathbb{P}^1$, Corollary \ref{delta is indecomposable cycle.} implies that their pullbacks $\widetilde{\gamma_{ijk, l}}$ are also $\mathbb{R}$-regulator indecomposable. When $d=4$, similarly Corollary \ref{gamma is indecomposable cycle.} implies the same indecomposability of $\widetilde{\delta_{i, lmn}}$.
Thus, by the discussion at the beginning of this section,
\begin{cor}
	The class of each algebraic cycle $(\mathcal{C}_ {ijk,l})_s$ in $\Griff^2(Y_s)$ is non-trivial for a general $s$.
	When $d = 4$, it also holds for $(\mathcal{D}_{i,lmn})_s$.
\end{cor}

\begin{figure}[H]
  \centering
  \includegraphics[width=10cm]{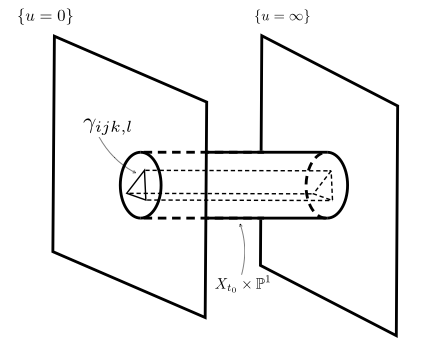}
    \caption{Central fiber $Y'_0$ of the family $\mathcal{Y}'$}
\end{figure}

\begin{proof}[Proof of Theorem \ref{non-triviality in Griffiths groups}]
	Since the construction of $\mathcal{C}_ {ijk,l}$ and $\mathcal{D}_{i,lmn}$ are exactly the same after changing $L$ with $M$, we only consider $\mathcal{C}_ {ijk,l}$.
	Recall that $\gamma_{ijk,l}$ is the sum of precycles with the form 
	\[
	\Gamma_{\alpha l} = (\frac{L_{\sigma (\alpha)}}{L_{\sigma^2 (\alpha)}}, \mathcal{L}_\alpha \cap \mathcal{M}_l). 
	\]
	For each of these precycles, we define a precycle $\Delta_{\alpha l}$ on $\mathcal{Y}$ as follows:
	Since we are taking the successive blow up along $\mathcal{L}_i \cap \mathcal{M}_l$, $\mathcal{L}_j \cap \mathcal{M}_l$, and then $\mathcal{L}_k \cap \mathcal{M}_l$, the zero locus $\mathcal{L}_\alpha \cap \mathcal{M}_l \cap \mathcal{Y} \cap \{u = 0\}$ with the original equation of $L_i$ and $M_j$ defines the unique irreducible component for $\alpha = i$, but two components for $\alpha = j, k$.
	In fact, each intersection $\mathcal{L}_\alpha \cap \mathcal{L}_{\sigma(\alpha)}$ in the above zero locus defines an exceptional curve after the strict transformation.
	(See the local explanation below in this proof).
	Hence the equation $\mathcal{L}_\alpha \cap \mathcal{M}_l \cap \mathcal{Y} \cap \{u = 0\}$ on $\mathcal{Y}$ is generally given by (a strict transformation of) $\mathbb{P}^1 \times \mathbb{P}^1$,  and the exceptional curve $\mathbb{P}^1$ if $\alpha = j,k$. We denote the former irreducible component by $Z_{\alpha l}$, which can be considered a reduced algebraic cycle.
	Therefore we can define a precycle 
	\[
	\Delta_{\alpha l} \coloneqq (\frac{L_{\sigma (\alpha)}}{L_{\sigma^2 (\alpha)}}, Z_{\alpha l}). 
	\]
		Then the restriction of this precycle on the component $X_{t_0} \times \mathbb{P}^1 \subset Y_0$ is the strict transformation $\widetilde{\Gamma_{\alpha l}}$ of the original $\Gamma_{\alpha l}$ on $X_{t_0} \times \{ u = 0 \}$.
		Hence the fiber $(\Delta_{ijk,l})_0$ of  $ \Delta_{ijk,l} \coloneqq \Delta_{i l} + \Delta_{j l} + \Delta_{k l}$ at $s = 0$ is an element of $ Z^2_{\#}(Y_0^{[0]},1)$ such that $\partial_{\mathcal{I}}(\Delta_{ijk,l})_0 = \widetilde{\Gamma_{i l}} + \widetilde{\Gamma_{j l}} + \widetilde{\Gamma_{k l}} = \widetilde{\gamma_{ijk, l}}$.
		Therefore we should define $\mathcal{C}_ {ijk,l}$ by
		\[\mathcal{C}_ {ijk,l} \coloneqq \partial_{\mathcal{B}}(\Delta_{ijk,l}).\]
To see that $\mathcal{C}_ {ijk,l}$ is in $CH^2_{hom}(Y_s)$ for a general $s$, we describe 
 locally what the algebraic cycle $\mathcal{C}_ {ijk,l}$ is.
By changing the coordinates of $\mathbb{P}^3$, assume $x = L_i, y = L_j, z = M_s$.
Then, with an invertible function $f$,  locally $\mathcal{Y}$ is given by the strict transformation of $(xy + t_0 f z)u + s(xy (u^2-1) + t_0 f z u) =0$ for the blow up along $L_i = M_s = 0$ and then $L_j = \widetilde{M_s} = 0$.
Denoting the blow up coordinates by $[X:Z]$ for the former one and $[Y: \widetilde{Z}]$ for the latter, specifically $\mathcal{Y}$ is given by the system of equations.
\[
\begin{cases}
	(XY + t_0 f \widetilde{Z}) u + s (XY(u^2 -1) + t_0 f \widetilde{Z} u) =0 \\
	xZ = zX \\
	y\widetilde{Z} = Z Y
\end{cases}
\]
Over these equations, $\Delta_{ijk,l}$ is given by $\Delta_{il}$ and $\Delta_{jl}$, which are defined by
\[
\Delta_{il} = (y, \{x = z = u = 0\}),\ 
\Delta_{jl} = (\frac{1}{x}, \{y = Z = u = 0\}).
\]
Note that their support cycles are the (blow up of) $\mathbb{P}^1 \times \mathbb{P}^1$ only when $s = 0$.
If $s \neq 0$, we need to take the intersection with $sXY=0$ additionally.
In these local coordinates, the boundary $\mathcal{C}_ {ijk,l} \coloneqq \partial_{\mathcal{B}}(\Delta_{ijk,l})$ is given by the algebraic cycle
\[[\{ x = y = z = u = 0\}] - [\{x = y = Z = u = 0\}],\]
which is exactly the exceptional curve $\mathbb{P}^1$ parametrized by $[X :Z]$.
Denoting this $\mathbb{P}^1$ by $P_{ij}$, therefore globally we obtain 
\[ \mathcal{C}_ {ijk,l} = P_{ij} + P_{jk} - P_{ik} \]
as Figure 4.
Since this cycle is in the exceptional curves for the blow up, it is homologus to zero. In other words, $\mathcal{C}_ {ijk,l} \in CH^2_{hom}(Y_s).$
The boundary $\mathcal{C}_{ijk,l}$ itself is on each fiber $Y_s$, but the precycle $\Delta_{ijk,l}$ is only on $Y_0$.
Hence the class of the higher cycle $\gamma_{ijk,l}$ ``goes down'' to $\mathcal{C}_{ijk,l}$  by the $K$-theory elevator on the singular fiber.

\end{proof}

\begin{figure}[H]
  \centering
  \includegraphics[width=8cm]{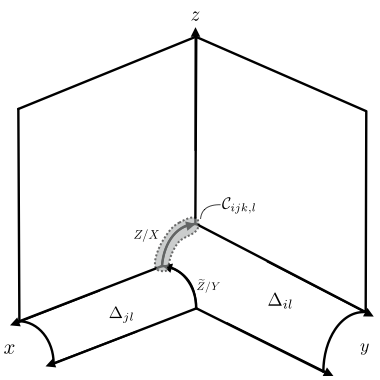}
    \caption{Local figure of $\Delta_{ijk,l}$}
\end{figure}
\begin{figure}[H]
  \centering
  \includegraphics[width=8cm]{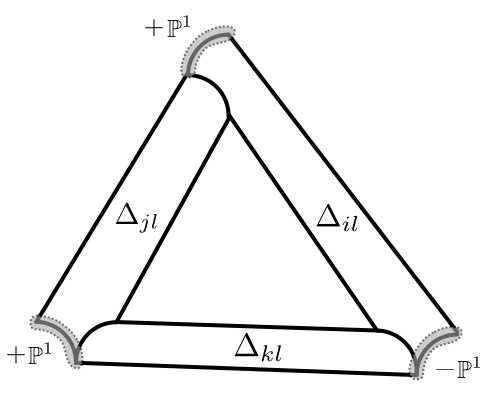}
    \caption{Algebraic Cycle $\mathcal{C}_ {ijk,l}$}
\end{figure}
\bibliographystyle{alpha}
\bibliography{Text}

\Address

\end{document}